\documentclass[letterpaper]{article}

\usepackage{charter}
\usepackage[T1]{fontenc}
\usepackage[utf8]{inputenc}
\usepackage[dvipsnames]{xcolor}
\usepackage[bookmarksnumbered]{hyperref}
\usepackage[inline]{enumitem}
\usepackage[super]{nth}
\usepackage[gen]{eurosym}
\usepackage{graphicx}
\usepackage{amssymb}
\usepackage{amsmath}
\usepackage{stmaryrd}
\usepackage{algorithm}
\usepackage{algorithmic}
\usepackage{dsfont}
\usepackage{mathtools}
\usepackage{tikz}
\usepackage[numbers]{natbib}
\usepackage{booktabs}
\usepackage{amsfonts}
\usepackage{amsthm}
\usepackage[font=small]{caption}
\usepackage{titlesec}
\usepackage{fancyhdr}
\usepackage{lipsum}
\usepackage[framemethod=TikZ]{mdframed}
\usepackage{pifont}
\usepackage{pdflscape}
\usepackage{dsfont}
\usepackage{parskip}

\renewcommand{\leq}{\leqslant}
\renewcommand{\geq}{\geqslant}
\renewcommand{\epsilon}{\varepsilon}

\makeatletter
\newenvironment{oracle}[1][htb]{%
    \renewcommand{\ALG@name}{Oracle}
   \begin{algorithm}[#1]%
  }{\end{algorithm}}
\makeatother

\makeatletter
\newcommand{\myitem}[1]{%
\item[#1]\protected@edef\@currentlabel{#1}%
}
\makeatother

\DeclareMathOperator*{\argmin}{\mathrm{arg\,min}}

\definecolor{crimson}{rgb}{0.86, 0.08, 0.24}
\definecolor{dodgerblue}{rgb}{0.12, 0.56, 1.0}
\definecolor{applegreen}{rgb}{0.55, 0.71, 0.0}
\definecolor{amethyst}{rgb}{0.6, 0.4, 0.8}
\definecolor{forestgreen}{rgb}{0.13, 0.55, 0.13}
\definecolor{burntorange}{rgb}{0.8, 0.33, 0.0}
\definecolor{deepcarrotorange}{rgb}{0.91, 0.41, 0.17}

\hypersetup{
    colorlinks=true,
    linkcolor={red!60!black},
    citecolor={green!40!black}
}

\newtheorem{theorem}{Theorem}[section]

\newtheorem{lemma}[theorem]{Lemma}

\newtheorem{remark}[theorem]{Remark}

\newtheorem{ctr}[theorem]{Counterexample}
\newtheorem{definition}[theorem]{Definition}
\newtheorem{assumption}[theorem]{Assumption}

\begin{document}

\begin{center}
 {\Large The Iterates of the Frank-Wolfe Algorithm May Not Converge}
\end{center}

\vspace{7mm}

\noindent
\textbf{J\'er\^ome Bolte}\textsuperscript{ 1}\hfill\href{jerome.bolte@ut-capitole.fr}{\ttfamily jerome.bolte@ut-capitole.fr}\\
\\
\textbf{Cyrille W.\ Combettes}\textsuperscript{ 1}\hfill\href{mailto:cyrille.combettes@tse-fr.eu}{\ttfamily cyrille.combettes@tse-fr.eu}\\
\\
\textbf{\'Edouard Pauwels}\textsuperscript{ 2 3}\hfill\href{edouard.pauwels@irit.fr}{\ttfamily edouard.pauwels@irit.fr}\\

\noindent
{\small\textsuperscript{1 }\emph{Toulouse School of Economics, Universit\'e Toulouse 1 Capitole, Toulouse, France}}\\
{\small\textsuperscript{2 }\emph{Centre National de la Recherche Scientifique, France}}\\
{\small\textsuperscript{3 }\emph{Institut de Recherche en Informatique de Toulouse, Universit\'e Toulouse 3 Paul Sabatier, Toulouse, France}}

\vspace{7mm}

\begin{center}
\begin{minipage}{0.85\textwidth}
\begin{center}
 \textbf{Abstract}
\end{center}
\vspace{3mm}

{\small The Frank-Wolfe algorithm is a popular method for minimizing a smooth convex function $f$ over a compact convex set $\mathcal{C}$. While many convergence results have been derived in terms of function values, hardly nothing is known about the convergence behavior of the sequence of iterates $(x_t)_{t\in\mathbb{N}}$. Under the usual assumptions, we design several counterexamples to the convergence of $(x_t)_{t\in\mathbb{N}}$, where $f$ is $d$-time continuously differentiable, $d\geq2$, and $f(x_t)\to\min_\mathcal{C}f$. Our counterexamples cover the cases of open-loop, closed-loop, and line-search step-size strategies. 
We do not assume \emph{misspecification} of the linear minimization oracle and our results thus hold regardless of the points it returns,
demonstrating the fundamental pathologies in the convergence behavior of $(x_t)_{t\in\mathbb{N}}$.}
\end{minipage}
\end{center}

\vspace{2mm}

\section{Introduction}
\label{sec:intro}

The Frank-Wolfe algorithm \cite{fw56}, a.k.a.\ conditional gradient algorithm \cite{levitin66}, addresses the 
optimization problem
\begin{align*}
 \min_{x\in\mathcal{C}}f(x),
\end{align*}
where $\mathcal{C}\subset\mathbb{R}^n$ is a nonempty compact convex set, $f\colon\mathcal{D}\to\mathbb{R}$ is a convex function with Lipschitz-continuous gradient, and $\mathcal{D}\subset\mathbb{R}^n$ is an open convex set containing $\mathcal{C}$. It does not require projections onto $\mathcal{C}$ to ensure the feasibility of its iterates and uses linear minimizations over $\mathcal{C}$ instead, which can be significantly cheaper to compute \cite{combettes21}. 
Another advantage is that it may generate iterates that are sparse with respect to the vertices of $\mathcal{C}$ \cite{clarkson10}. Both properties have led to many applications of the Frank-Wolfe algorithm \cite{leblanc75,hazan08,lacoste13svm,joulin15video,luise19,combettes22,do21}.

With a suitable step-size strategy, the sequence of function values $(f(x_t))_{t\in\mathbb{N}}$ converges to $\min_\mathcal{C}f$ at a rate $\mathcal{O}(1/t)$ \cite{levitin66,dunn78,jaggi13fw}, and this result has been extensively complemented: lower bounds have been established \cite{canon68,jaggi13fw,lan13}, faster rates under additional assumptions have been derived \cite{levitin66,guelat86,garber15,kerdreux21}, and variants speeding up the algorithm have been developed \cite{lacoste15,garber16dicg,braun19bcg,combettes20}. However, it is not known whether the sequence of iterates $(x_t)_{t\in\mathbb{N}}$ converges or not, apart from the trivial case where the solution set $\argmin_\mathcal{C}f$ is a singleton \cite{levitin66,dunn78}. This is in contrast with other popular convex optimization methods such as the gradient descent, coordinate descent, mirror descent, proximal point, forward-backward, or Douglas-Rachford algorithm, for which convergence results for the sequence of iterates have been proved \cite{cp11} or disproved \cite{bp22}.

Note that counterexamples to the convergence of $(x_t)_{t\in\mathbb{N}}$ could be designed trivially by assuming that the linear minimization oracle is \emph{misspecified} (Definition~\ref{def:mis}), i.e., that it does not necessarily return the same output each time the same input is entered.
A simple illustration is given  by  the minimization of $f=0$ over $\mathcal{C}=\left[0,1\right]$ using an open-loop strategy~\ref{step:open1}--\ref{step:open2}. As misspecification allows to move towards $0$ or towards $1$ at each iteration, $(x_t)_{t\in\mathbb{N}}$ may follow any arbitrary trajectory in $\left[0,1\right]$. In $\mathbb{R}^2$, one can also easily build counterexamples where $(x_t)_{t\in\mathbb{N}}$ does not converge while remaining out of $\argmin_\mathcal{C}f$, for any step-size strategy~\ref{step:open1}--\ref{step:ls}: consider, e.g., minimizing $f=\operatorname{dist}(\cdot,\left[(-1/2,0),(1/2,0)\right])^2$ over $\mathcal{C}=\operatorname{conv}\{(-2,1/4),(-1,0),(1,0),(0,1)\}$.
These counterexamples are artificial and somehow vain, as they rely on the misspecification assumption and a very specific adversarial choice of the points returned by the oracle. Our work does not assume misspecification of the oracle
to demonstrate that $(x_t)_{t\in\mathbb{N}}$ may not converge and
relies instead on the joint geometry of $\mathcal{C}$ and $f$. 

\paragraph{Contributions.} We show that the sequence of iterates $(x_t)_{t\in\mathbb{N}}$ generated by the Frank-Wolfe algorithm does not converge in general. We design several instances of $\mathcal{C}$ and $f$ satisfying the usual assumptions, i.e., $\mathcal{C}$ is compact and convex and $f$ is convex with Lipschitz-continuous gradient, for which the sequence $(x_t)_{t\in\mathbb{N}}$ 
does not converge, while $f(x_t)\to\min_\mathcal{C}f$. We cover all step-size strategies for which convergence results for the function values have been established, i.e., the so-called open-loop, closed-loop, and line-search strategies. Furthermore, the functions $f$ designed are $d$-time continuously differentiable, where $d\geq2$ is arbitrarily large. 
We use key results from \cite{bp22} 
and do not assume \emph{misspecification} of the linear minimization oracle, thus demonstrating the fundamental pathologies in the convergence behavior of $(x_t)_{t\in\mathbb{N}}$. 

\section{Preliminaries}

\subsection{Notation and definitions}

We consider the Euclidean space $(\mathbb{R}^n,\langle\cdot,\cdot\rangle)$ equipped with the Euclidean norm $\|\cdot\|$.
Given a convex set $\mathcal{C}\subset\mathbb{R}^n$ and $x\in\mathcal{C}$, the tangent cone to $\mathcal{C}$ at $x$ is $\mathcal{T}_\mathcal{C}x=\operatorname{cl}\{\lambda(y-x)\mid\lambda>0,y\in\mathcal{C}\}$ and the normal cone to $\mathcal{C}$ at $x$ is $\mathcal{N}_\mathcal{C}x=\{u\in\mathbb{R}^n\mid\forall y\in\mathcal{C},\langle y-x,u\rangle\leq0\}$. 
The diameter and the set of vertices of $\mathcal{C}$ are denoted by $\operatorname{diam}\mathcal{C}$ and
$\operatorname{vert}\mathcal{C}$ respectively.
The Hausdorff distance between two nonempty compact sets $\mathcal{A},\mathcal{B}\subset\mathbb{R}^n$ is
denoted by $\operatorname{dist}(\mathcal{A},\mathcal{B})$.
The canonical vectors in $\mathbb{R}^2$ are $e_1=(1,0)$ and $e_2=(0,1)$. The unit sphere in $\mathbb{R}^2$ is $\mathbb{S}=\{x\in\mathbb{R}^2\mid\|x\|=1\}$. 
Given three points $A,B,C\in\mathbb{R}^2$, the angle at $B$ formed by the rays $[BA)$ and $[BC)$ is denoted by $\widehat{ABC}$.

\subsection{The Frank-Wolfe algorithm}
\label{sec:fw}

The Frank-Wolfe algorithm \cite{fw56} is presented in Algorithm~\ref{fw}. It assumes access to a linear minimization oracle (\texttt{LMO}, Oracle~\ref{lmo}) solving linear minimization problems over the constraint set $\mathcal{C}$.

\begin{oracle}[h]
\caption{\texttt{LMO}}
\label{lmo}
\begin{algorithmic}[1]
 \REQUIRE Vector $u\in\mathbb{R}^n$.
 \ENSURE Point $v\in\argmin_{v\in\mathcal{C}}\langle v,u\rangle$. 
\end{algorithmic}
\end{oracle}

The linear minimization oracle does not ensure a priori that the same output is returned each time the same input is entered. When $\argmin_{v\in\mathcal{C}}\langle v,u\rangle$ is not a singleton, it makes a choice on which element to return, based on, e.g., a proximity score, an ordering of vertices, or a random seed.
This leads to the notion of \emph{misspecification} for the oracle.

\begin{definition}
 \label{def:mis}
 The linear minimization oracle \emph{\texttt{LMO}} is specified if $u\in\mathbb{R}^n\mapsto\emph{\texttt{LMO}}(u)$ is a well-defined function and satisfies $\emph{\texttt{LMO}}(u)=\emph{\texttt{LMO}}(u')$ for all positively colinear $u,u'\in\mathbb{R}^n$. Otherwise, it is misspecified.
\end{definition}

\begin{algorithm}[h]
\caption{Frank-Wolfe}
\label{fw}
\begin{algorithmic}[1]
\REQUIRE Start point $x_0\in\mathcal{C}$, step-size strategy $(\gamma_t)_{t\in\mathbb{N}}\in\left[0,1\right]^\mathbb{N}$.
\ENSURE Sequence of iterates $(x_t)_{t\in\mathbb{N}}\in\mathcal{C}^\mathbb{N}$.
\FOR{$t=0$ \textbf{to} $T-1$}
\STATE$v_t\gets\texttt{LMO}(\nabla f(x_t))$\label{fw:lmo} 
\STATE$x_{t+1}\gets(1-\gamma_t)x_t+\gamma_tv_t$\label{fw:new}
\ENDFOR
\end{algorithmic}
\end{algorithm}

\begin{remark}
\begin{enumerate*}[label=(\roman*)]
 \item The step-size strategy $(\gamma_t)_{t\in\mathbb{N}}$ need not be numerically defined before running Algorithm~\ref{fw}. More precisely, we may say that a sequence $(x_t)_{t\in\mathbb{N}}\in(\mathbb{R}^n)^\mathbb{N}$ is generated by the Frank-Wolfe algorithm if $x_0\in\mathcal{C}$ and for all $t\in\mathbb{N}$, there exist $v_t\in\argmin_{v\in\mathcal{C}}\langle v,\nabla f(x_t)\rangle$ and $\gamma_t\in\left[0,1\right]$ such that $x_{t+1}=(1-\gamma_t)x_t+\gamma_tv_t$. 
\item When $\argmin_{v\in\mathcal{C}}\langle v,\nabla f(x_t)\rangle$ is not a singleton, the choice for $v_t$ is made by the linear minimization oracle. If the oracle is misspecified, then $(v_t)_{t\in\mathbb{N}}$, and thus $(x_t)_{t\in\mathbb{N}}$, may not be uniquely determined by $x_0$ and $(\gamma_t)_{t\in\mathbb{N}}$.
\end{enumerate*}
\end{remark}

Since $x_{t+1}$ is generated by convex combination of $x_t$ and $v_t$ (Line~\ref{fw:new}), it is feasible and there is no need for a projection step back onto $\mathcal{C}$. 
Note that $x_{t+1}=x_t+\gamma_t(v_t-x_t)$ is obtained by moving from $x_t$ towards $v_t$.
Furthermore, assuming the linear minimization oracle returns only vertices of $\mathcal{C}$, at most one new vertex is added to the convex decomposition of $x_t$ to obtain $x_{t+1}$, making the sequence $(x_t)_{t\in\mathbb{N}}$ sparse with respect to the vertices of $\mathcal{C}$.

There are mainly four step-size strategies that have been considered and for which the rate at which $f(x_t)\to\min_\mathcal{C}f$ has been established (Theorem~\ref{th:fw}). These are:
\begin{enumerate}[label=(\roman*)]
 \item\label{step:open1} an open-loop strategy
 \begin{align*}
  \forall t\in\mathbb{N},\,
  \gamma_t=\frac{1}{t+1};
 \end{align*}

 \myitem{(i')}\label{step:open2} another open-loop strategy
 \begin{align*}
  \forall t\in\mathbb{N},\,
  \gamma_t=\frac{2}{t+2};
 \end{align*}

 \item\label{step:closed} the closed-loop strategy
 \begin{align*}
  \forall t\in\mathbb{N},\,
  \gamma_t=\min\left\{\frac{\langle x_t-v_t,\nabla f(x_t)\rangle}{L\|x_t-v_t\|^2},1\right\};
 \end{align*}

 \item\label{step:ls} the line-search strategy
 \begin{align*}
  \forall t\in\mathbb{N},\,\gamma_t\in\argmin_{\gamma\in\left[0,1\right]}f(x_t+\gamma(v_t-x_t)).
 \end{align*}
\end{enumerate}

Assumption~\ref{aspt} collects the general set of assumptions for the Frank-Wolfe algorithm.

\begin{assumption}
 \label{aspt}
 Let $\mathcal{C}\subset\mathbb{R}^n$ be a nonempty compact convex set and $f\colon\mathcal{D}\to\mathbb{R}$ be a convex function with $L$-Lipschitz-continuous gradient, where $L>0$ and $\mathcal{D}\subset\mathbb{R}^n$ is an open convex set containing $\mathcal{C}$. 
\end{assumption}

\begin{theorem}[\cite{levitin66,dunn78,jaggi13fw}]
 \label{th:fw}
 Consider Assumption~\ref{aspt}.
 Then the Frank-Wolfe algorithm satisfies for all $t\geq1$,
 \begin{align*}
  f(x_t)-\min_\mathcal{C}f\leq
  \begin{cases}
   L(\operatorname{diam}\mathcal{C})^2(1+\ln t)/(2t)&\text{if }(\gamma_t)_{t\in\mathbb{N}}\text{ follows~\ref{step:open1}};\\
   2L(\operatorname{diam}\mathcal{C})^2/(t+2)&\text{if }(\gamma_t)_{t\in\mathbb{N}}\text{ follows~\ref{step:open2}};\\
   4L(\operatorname{diam}\mathcal{C})^2/(t+2)&\text{if }(\gamma_t)_{t\in\mathbb{N}}\text{ follows~\ref{step:closed}};\\
   4L(\operatorname{diam}\mathcal{C})^2/(t+2)&\text{if }(\gamma_t)_{t\in\mathbb{N}}\text{ follows~\ref{step:ls}}.
  \end{cases}
 \end{align*}
\end{theorem}

\section{The Frank-Wolfe algorithm may not converge}
\label{sec:ctr}

We turn on to the main results of the paper. We 
design several instances of $\mathcal{C}$ and $f$ satisfying Assumption~\ref{aspt} and for which the sequence of iterates $(x_t)_{t\in\mathbb{N}}$ generated by the Frank-Wolfe algorithm using a usual step-size strategy~\ref{step:open1}--\ref{step:ls} does not converge, while Theorem~\ref{th:fw} still holds. The strength of these counterexamples is to be valid against specified linear minimization oracles, demonstrating that nonconvergence is intrinsic to the Frank-Wolfe algorithm in the smooth convex optimization setting. 
We will use a key result from \cite{bp22} which we restate in Theorem~\ref{th:counter}. It shows that we can design a convex function, arbitrarily smooth, from polygonal sketches, and that we can further choose its gradient directions at specific points from a set of directions described by the geometry of the sketches.

Given a polytope $\mathcal{P}\subset\mathbb{R}^2$ and a vertex $V\in\operatorname{vert}\mathcal{P}$, we denote by $\displaystyle \mathcal{K}_\mathcal{P}V=\mathcal{N}_\mathcal{P}V\cap\left(-\mathcal{T}_\mathcal{P}V\right)$ the cone of admissible directions at $V$ for $\mathcal{P}$.

\begin{theorem}[\cite{bp22}]
 \label{th:counter} 
 In $\mathbb{R}^2$, let $(\mathcal{P}_\ell)_{\ell\in\mathbb{N}}$ be a sequence of polytopes, $(u_k)_{k\in\mathbb{N}}\in\mathbb{S}^\mathbb{N}$ be a sequence of unit vectors, $(V_k)_{k\in\mathbb{N}}\in\left(\bigcup_{\ell\in\mathbb{N}}\operatorname{vert}\mathcal{P}_\ell\right)^\mathbb{N}$ be a sequence of vertices of the polytopes, and $(\delta_\ell)_{\ell\in\mathbb{N}}\in\left]0,1\right[^\mathbb{N}$ be such that 
 \begin{enumerate}[label=(\roman*)]
  \item for all $k\in\mathbb{N}$, there exists $\ell\in\mathbb{N}$ such that $u_k\in\mathcal{K}_{\mathcal{P}_\ell}V_k$; 
  \item for all $\ell\in\mathbb{N}$ and $\lambda\in\left[1-\delta_\ell,1+\delta_\ell\right]$, $\lambda\mathcal{P}_{\ell+1}\subset\operatorname{int}\mathcal{P}_\ell$;
  \item for all $\ell\in\mathbb{N}$, $0\in\operatorname{int}\mathcal{P}_\ell$.
 \end{enumerate}
 Then for any open convex set $\mathcal{D}\subset\mathbb{R}^2$ and $d\in\mathbb{N}$ such that $\mathcal{P}_0\subset\mathcal{D}$ and $d\geq2$ respectively, there exist a $d$-time continuously differentiable convex function $f\colon\mathcal{D}\to\mathbb{R}$ and a sequence $(\eta_\ell)_{\ell\in\mathbb{N}}\in\mathbb{R}^\mathbb{N}$ such that for all $k,\ell\in\mathbb{N}$, 
 \begin{enumerate}[label=(\roman*)]
  \setcounter{enumi}{3}
  \item $\eta_{\ell+1}<\eta_\ell$;
  \item $\mathcal{P}_\ell\subset\{x\in\mathbb{R}^2\mid f(x)\leq\eta_\ell\}$;
  \item $f(V)=\eta_\ell$ for all $V\in\operatorname{vert}\mathcal{P}_\ell$;
  \item $\operatorname{dist}(\mathcal{P}_\ell,\{x\in\mathbb{R}^2\mid f(x)\leq\eta_\ell\})\leq\delta_\ell$;
  \item\label{curv} $\{x\in\mathbb{R}^2\mid f(x)\leq\eta_\ell\}$ has positive curvature;
  \item\label{coli} $\nabla f(V_k)$ is positively colinear to $u_k$;
  \item $\nabla^2f$ is positive definite outside of $\argmin_\mathcal{D}f=\bigcap_{\ell\in\mathbb{N}}\mathcal{P}_\ell$.
 \end{enumerate}
\end{theorem}

\begin{remark}
 In Theorem~\ref{th:counter}, $\nabla f$ is Lipschitz-continuous on $\mathcal{D}$ by the mean value theorem, since $d\geq2$. Thus, $f$ satisfies Assumption~\ref{aspt}.
\end{remark}

\newpage
\begin{lemma}
 \label{cor:ls}
 Consider the function $f$ and all other variables defined in Theorem~\ref{th:counter}. Let $\mathcal{C}\subset\mathcal{D}$ be a compact convex set and consider minimizing $f$ over $\mathcal{C}$ using the Frank-Wolfe algorithm with the line-search strategy~\ref{step:ls}. 
 Let $t\in\mathbb{N}$ and suppose that there exists $k\in\mathbb{N}$ such that
 $V_k\in\left[x_t,v_t\right]$ and $\langle v_t-x_t,u_k\rangle=0$. Then $x_{t+1}=V_k$.
\end{lemma}

\begin{proof}
 We have $x_{t+1}=(1-\gamma_t)x_t+\gamma_tv_t$ where $\gamma_t\in\argmin_{\gamma\in\left[0,1\right]}f(x_t+\gamma(v_t-x_t))$, i.e., $x_{t+1}\in\left[x_t,v_t\right]$ and $\langle v_t-x_t,\nabla f(x_{t+1})\rangle=0$. By Theorem~\ref{th:counter}\ref{coli}, $u_k$ and $\nabla f(V_k)$ are positively colinear so $\langle v_t-x_t,\nabla f(V_k)\rangle=0$. By Theorem~\ref{th:counter}\ref{curv}, $\{x\in\mathbb{R}^2\mid f(x)\leq f(V_k)\}$ has positive curvature, which yields $x_{t+1}=V_k$.
\end{proof}

\begin{lemma}
 \label{lem:align}
 Consider the function $f$ and all other variables defined in Theorem~\ref{th:counter}. Let $k_1,k_2,\ell_1,\ell_2\in\mathbb{N}$, $\ell_1>\ell_2$, and $u\in\mathbb{S}$, and suppose that $V_{k_1}\in\mathcal{P}_{\ell_1}$ and $V_{k_2}\in\mathcal{P}_{\ell_2}$ are such that $0$, $V_{k_1}$, $V_{k_2}$ are aligned and $u_{k_1}=u_{k_2}=u$. Then $\nabla f(x)$ is positively colinear to $u$ for all $x\in\left[V_{k_1},V_{k_2}\right]$.
\end{lemma}

\begin{proof}
 Let $\mathcal{L}_1=\{x\in\mathbb{R}^2\mid f(x)\leq f(V_{k_1})\}$ and $\mathcal{L}_2=\{x\in\mathbb{R}^2\mid f(x)\leq f(V_{k_2})\}$. By Theorem~\ref{th:counter}\ref{curv}, they have positive curvature. 
 Let $n_1\colon\operatorname{bd}\mathcal{L}_1\to\mathbb{S}$ and $n_2\colon\operatorname{bd}\mathcal{L}_2\to\mathbb{S}$ be their Gauss maps respectively, which are diffeomorphisms that map a point on the boundary of the set to the normal to the set at that point \cite[Sec.~2.5]{schneider13}. By Theorem~\ref{th:counter}\ref{coli}, $n_1(V_{k_1})=u=n_2(V_{k_2})$.
 Let $x'\in\left[V_{k_1},V_{k_2}\right]$ and $\mathcal{L}=\{x\in\mathbb{R}^2\mid f(x)\leq f(x')\}$. 
 The construction from \cite{bp22} leading to Theorem~\ref{th:counter} gives $\mathcal{L}=\alpha\mathcal{L}_1+\beta\mathcal{L}_2$, where $\alpha,\beta\geq0$ and $\alpha+\beta>0$. Thus, by alignment with $0$, $x'=\alpha V_{k_1}+\beta V_{k_2}$, and, by \cite[Lem.~2]{bp22}, $\mathcal{L}$ has positive curvature. 
 Let $n\colon\operatorname{bd}\mathcal{L}\to\mathbb{S}$ be its Gauss map. By \cite[Lem.~2]{bp22} again, $n^{-1}(u)=\alpha n_1^{-1}(u)+\beta n_2^{-1}(u)=\alpha V_{k_1}+\beta V_{k_2}=x'$, i.e, $n(x')=u$. 
\end{proof}

\subsection{Counterexample 1: Line-search strategy}
\label{sec:ls}

We start with a counterexample for the line-search strategy~\ref{step:ls}.

\begin{ctr}
 There exist $\mathcal{C}$ and $f$, satisfying Assumption~\ref{aspt}, and $x_0\in\mathcal{C}$ such that any sequence $(x_t)_{t\in\mathbb{N}}$ generated by Frank-Wolfe algorithm using the line-search strategy~\ref{step:ls} does not converge.
\end{ctr}

Let $\mathcal{C}=\operatorname{conv}\{(-1,0),(0,1),(1,0)\}$ and $f$ be defined via Theorem~\ref{th:counter} with $A_0=(-1/2,0)$, $B_0=(-1/4,3/4)$, $C_0=(1/4,3/4)$, $D_0=(1/2,0)$, 
and, for all $k\in\mathbb{N}$,
\begin{align*}
 &A_{k+1}=\displaystyle\left(-\frac{1}{4}+\frac{\langle B_{k+1},e_2\rangle}{\langle B_k,e_2\rangle}\left(\langle A_k,e_1\rangle+\frac{1}{4}\right),0\right),\\
 &B_{k+1}=\displaystyle\left(-\frac{1}{4},\frac{3}{5}\langle C_k,e_2\rangle\right),\\
 &C_{k+1}=\displaystyle\left(\frac{1}{4},\frac{3}{5}\langle B_k,e_2\rangle\right),\\
 &D_{k+1}=\displaystyle\left(\frac{1}{4}+\frac{\langle C_{k+1},e_2\rangle}{\langle C_k,e_2\rangle}\left(\langle D_k,e_1\rangle-\frac{1}{4}\right),0\right),
\end{align*}
and 
 $\mathcal{P}_k=\operatorname{conv}\{A_k,B_k,C_k,D_k,-D_k,-C_k,-B_k,-A_k\}$. 
Then $\mathcal{C}$ and $f$ satisfy Assumption~\ref{aspt} so Theorem~\ref{th:fw} holds.
The solution set is $\argmin_\mathcal{C}f=\bigcap_{k\in\mathbb{N}}\mathcal{P}_k=\left[(-1/4,0),(1/4,0)\right]$.
The construction is such that for all $k\in\mathbb{N}$, the points $B_k$, $C_{k+1}$, $(1,0)$, and the points $C_k$, $B_{k+1}$, $(-1,0)$, are aligned, and $\widehat{A_kB_kC_k}=\widehat{A_0B_0C_0}$ and $\widehat{B_kC_kD_k}=\widehat{B_0C_0D_0}$.
An illustration is presented in Figure~\ref{fig:ctr2}.

\begin{figure}[h]
\vspace{2mm}
 \centering 
 \begin{tikzpicture}[thick, scale=4]
  
  \coordinate (x0) at (0, 1);
  \coordinate (v1) at (-1, 0);
  \coordinate (v2) at (1, 0);
  
  \coordinate (a1) at (0.25, 0.75);
  \coordinate (b1) at (0.5, 0);
  \coordinate (c1) at (-0.5, 0);
  \coordinate (d1) at (-0.25, 0.75);
  
  \coordinate (a2) at (0.25, 0.45);
  \coordinate (b2) at (0.4, 0);
  \coordinate (c2) at (-0.4, 0);
  \coordinate (d2) at (-0.25, 0.45);
  
  \coordinate (a3) at (0.25, 0.27);
  \coordinate (b3) at (0.34, 0);
  \coordinate (c3) at (-0.34, 0);
  \coordinate (d3) at (-0.25, 0.27);
  
  \coordinate (a4) at (0.25, 0.162);
  \coordinate (b4) at (0.304, 0);
  \coordinate (c4) at (-0.304, 0);
  \coordinate (d4) at (-0.25, 0.162);
  
  \coordinate (a5) at (0.25, 0.0972);
  \coordinate (b5) at (0.2824, 0);
  \coordinate (c5) at (-0.2824, 0);
  \coordinate (d5) at (-0.25, 0.0972);
  
  \coordinate (a6) at (0.25, 0.0583);
  \coordinate (b6) at (0.2694, 0);
  \coordinate (c6) at (-0.2694, 0);
  \coordinate (d6) at (-0.25, 0.0583);
  
  \coordinate (a7) at (0.25, 0.0350);
  \coordinate (b7) at (0.2617, 0);
  \coordinate (c7) at (-0.2617, 0);
  \coordinate (d7) at (-0.25, 0.0350);
  
  \draw[blue!30, dotted] (d1) -- (1, 0);
  \draw[blue!30, dotted] (a1) -- (-1, 0);
  \draw[blue!30, dotted] (d2) -- (1, 0);
  \draw[blue!30, dotted] (a2) -- (-1, 0);
  \draw[blue!30, dotted] (d3) -- (1, 0);
  \draw[blue!30, dotted] (a3) -- (-1, 0);
  \draw[blue!30, dotted] (d4) -- (1, 0);
  \draw[blue!30, dotted] (a4) -- (-1, 0);
    
  \draw[orange] (c1) -- (d1) -- (a1) -- (b1);
  \draw[orange] (c2) -- (d2) -- (a2) -- (b2);
  \draw[orange] (c3) -- (d3) -- (a3) -- (b3);
  \draw[orange] (c4) -- (d4) -- (a4) -- (b4);
  \draw[orange] (c5) -- (d5) -- (a5) -- (b5);
  \draw[orange] (c6) -- (d6) -- (a6) -- (b6);
  \draw[orange] (c7) -- (d7) -- (a7) -- (b7);
  
\draw (-1/4, 0) -- (v1) -- (x0) -- (v2) -- (1/4, 0);
  \draw[crimson] (-1/4, 0) -- (1/4, 0);

  \draw[fill=black] (c1) circle (0.23pt) node [below] {$A_0\quad$};
  \draw[fill=black] (d1) circle (0.23pt) node [above] {$B_0\quad$};
  \draw[fill=black] (a1) circle (0.23pt) node [above] {$\quad C_0$};
  \draw[fill=black] (b1) circle (0.23pt) node [below] {$\quad D_0$};
  
  \draw[fill=black] (c2) circle (0.23pt) node [below] {$A_1$};
  \draw[fill=black] (d2) circle (0.23pt) node [above] {$B_1$};
  \draw[fill=black] (a2) circle (0.23pt) node [above] {$C_1$};
  \draw[fill=black] (b2) circle (0.23pt) node [below] {$D_1$};
  
  \draw[fill=black] (d3) circle (0.23pt);
  \draw[fill=black] (a3) circle (0.23pt);
 
  \draw[fill=black] (d4) circle (0.23pt);
  \draw[fill=black] (a4) circle (0.23pt);
  \draw[fill=black] (d5) circle (0.23pt);
  \draw[fill=black] (a5) circle (0.23pt);
 
  \draw[fill=black] (v1) circle (0.23pt) node [below] {$(-1,0)$};
  \draw[fill=black] (v2) circle (0.23pt) node [below] {$(1,0)$};
  
 \end{tikzpicture}
 \caption{The constraint set (in black), polygonal sketches of the objective function (in orange), and the solution set (in red).
 For all $k\in\mathbb{N}$, the points $B_k$, $C_{k+1}$, $(1,0)$, and the points $C_k$, $B_{k+1}$, $(-1,0)$, are aligned, and $\widehat{A_kB_kC_k}=\widehat{A_0B_0C_0}$ and $\widehat{B_kC_kD_k}=\widehat{B_0C_0D_0}$.}
 \label{fig:ctr2}
\end{figure}
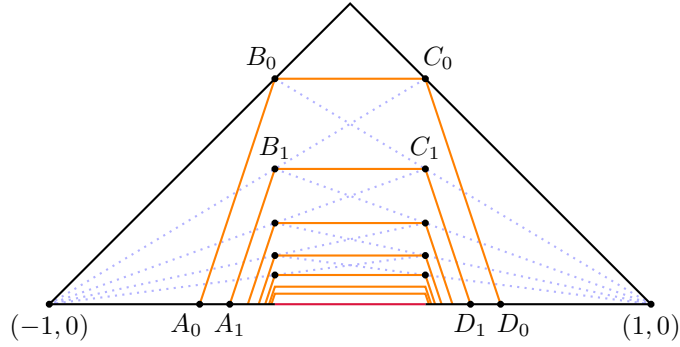

The construction allows to choose $(u_k)_{k\in\mathbb{N}}$ such that
$u_0\in\mathcal{K}_{\mathcal{P}_0}C_0$ and, for all $k\geq1$,
\begin{align*}
 \begin{cases}
  u_k\in\mathcal{K}_{\mathcal{P}_k}C_k\text{ and }\langle C_k-B_{k-1},u_k\rangle=0&\text{if }k\text{ is even};\\
  u_k\in\mathcal{K}_{\mathcal{P}_k}B_k\text{ and }\langle B_k-C_{k-1},u_k\rangle=0&\text{if }k\text{ is odd}.
 \end{cases}
\end{align*}
Let $x_0=C_0$. By induction and using Lemma~\ref{cor:ls}, we obtain for all $t\in\mathbb{N}$,
\begin{align*}
 x_t=
 \begin{cases}
  C_t&\text{if }t\text{ is even};\\
  B_t&\text{if }t\text{ is odd}.
 \end{cases}
\end{align*}
Therefore, $\langle x_t,e_1\rangle
 =(-1)^t/4$ for all $t\in\mathbb{N}$, so $(x_t)_{t\in\mathbb{N}}$ does not converge.
An illustration is presented in Figure~\ref{fig:ls}.

\begin{figure}[h]
\vspace{2mm}
 \centering 
 \begin{tikzpicture}[thick, scale=4]
  
  \coordinate (x0) at (0, 1);
  \coordinate (v1) at (-1, 0);
  \coordinate (v2) at (1, 0);
  
  \coordinate (a1) at (0.25, 0.75);
  \coordinate (b1) at (0.5, 0);
  \coordinate (c1) at (-0.5, 0);
  \coordinate (d1) at (-0.25, 0.75);
  
  \coordinate (a2) at (0.25, 0.45);
  \coordinate (b2) at (0.4, 0);
  \coordinate (c2) at (-0.4, 0);
  \coordinate (d2) at (-0.25, 0.45);
  
  \coordinate (a3) at (0.25, 0.27);
  \coordinate (b3) at (0.34, 0);
  \coordinate (c3) at (-0.34, 0);
  \coordinate (d3) at (-0.25, 0.27);
  
  \coordinate (a4) at (0.25, 0.162);
  \coordinate (b4) at (0.304, 0);
  \coordinate (c4) at (-0.304, 0);
  \coordinate (d4) at (-0.25, 0.162);
  
  \coordinate (a5) at (0.25, 0.0972);
  \coordinate (b5) at (0.2824, 0);
  \coordinate (c5) at (-0.2824, 0);
  \coordinate (d5) at (-0.25, 0.0972);
  
  \coordinate (a6) at (0.25, 0.0583);
  \coordinate (b6) at (0.2694, 0);
  \coordinate (c6) at (-0.2694, 0);
  \coordinate (d6) at (-0.25, 0.0583);
  
  \coordinate (a7) at (0.25, 0.0350);
  \coordinate (b7) at (0.2617, 0);
  \coordinate (c7) at (-0.2617, 0);
  \coordinate (d7) at (-0.25, 0.0350);
    
  \draw[orange!50] (c1) -- (d1) -- (a1) -- (b1); 
  \draw[orange!50] (c2) -- (d2) -- (a2) -- (b2);
  \draw[orange!50] (c3) -- (d3) -- (a3) -- (b3);
  \draw[orange!50] (c4) -- (d4) -- (a4) -- (b4);
  \draw[orange!50] (c5) -- (d5) -- (a5) -- (b5);
  \draw[orange!50] (c6) -- (d6) -- (a6) -- (b6);
  \draw[orange!50] (c7) -- (d7) -- (a7) -- (b7);

  \draw[->][ForestGreen] (0.25, 0.75) -- (0.25+0.04, 0.75+0.0916);
  \draw[->][ForestGreen] (-0.25, 0.45) -- (-0.25-0.0514, 0.45+0.0857);
  \draw[->][ForestGreen] (0.25, 0.27) -- (0.25+0.0339, 0.27+0.0941);
  \draw[->][ForestGreen] (-0.25, 0.162) -- (-0.25-0.0211, 0.162+0.0977);
  \draw[->][ForestGreen] (0.25, 0.0972) -- (0.25+0.0128, 0.0972+0.0992);
  \draw[->][ForestGreen] (-0.25, 0.0583) -- (-0.25-0.0078, 0.0583+0.0997);
  \draw[->][ForestGreen] (0.25, 0.0350) -- (0.25+0.0046, 0.0350+0.0999);

  \draw[->][blue] (a1) -- (d2);
  \draw[->][blue] (d2) -- (a3);
  \draw[->][blue] (a3) -- (d4);
  \draw[->][blue] (d4) -- (a5);
  \draw[->][blue] (a5) -- (d6);
  \draw[->][blue] (d6) -- (a7);
  
\draw (-1/4, 0) -- (v1) -- (x0) -- (v2) -- (1/4, 0);
  \draw[crimson] (-1/4, 0) -- (1/4, 0);

  \draw[fill=black] (a1) circle (0.23pt) node [right] {$x_0$};
  \draw[fill=black] (d2) circle (0.23pt);
  \draw[fill=black] (a3) circle (0.23pt);
  \draw[fill=black] (d4) circle (0.23pt);
  \draw[fill=black] (a5) circle (0.23pt);
  \draw[fill=black] (d6) circle (0.23pt);
  \draw[fill=black] (a7) circle (0.23pt);
  
 \end{tikzpicture}
 \caption{The constraint set (in black), polygonal sketches of the objective function (in orange), the solution set (in red), gradient directions (in green), and the trajectory (in blue) of the sequence $(x_t)_{t\in\mathbb{N}}$ generated by the Frank-Wolfe algorithm with the line-search strategy~\ref{step:ls} starting from $x_0$. The abscisse of $x_t$ is $(-1)^t/4$ for all $t\in\mathbb{N}$, so $(x_t)_{t\in\mathbb{N}}$ does not converge.}
 \label{fig:ls}
\end{figure}

\subsection{Counterexample 2: Line-search strategy and solution set in the interior}
\label{sec:int}

In our next counterexample, the solution set $\argmin_\mathcal{C}f$ lies in the interior of the constraint set. 

\begin{ctr}
 There exist $\mathcal{C}$ and $f$, satisfying Assumption~\ref{aspt} and $\argmin_\mathcal{C}f\subset\operatorname{int}\mathcal{C}$, and $x_0\in\mathcal{C}$ such that any sequence $(x_t)_{t\in\mathbb{N}}$ generated by Frank-Wolfe algorithm using the line-search strategy~\ref{step:ls} does not converge.
\end{ctr}

Let $\mathcal{C}=\operatorname{conv}\{(-1,-1),(-1,1),(1,1),(1,-1)\}$ and $f$ be defined via Theorem~\ref{th:counter} with $A_0=(-1/10,-1)$, $B_0=(-1,1/10)$, $C_0=(1/10,1)$, $D_0=(1,-1/10)$, $\lambda_0=1$, and, for all $k\in\mathbb{N}$, $\lambda_{k+1}=110\lambda_k/(90+101\lambda_k)$, 
\begin{align*}
 A_k=\lambda_kA_0,\quad
 B_k=\lambda_kB_0,\quad
 C_k=\lambda_kC_0,\quad
 D_k=\lambda_kD_0,
\end{align*}
and $\mathcal{P}_k=\operatorname{conv}\{\lambda_kA_0,\lambda_kB_0,\lambda_kC_0,\lambda_kD_0\}=\lambda_k\mathcal{P}_0$. 
Thus, $(\lambda_k)_{k\in\mathbb{N}}$ is a decreasing sequence that converges to $1/5$. Then $\mathcal{C}$ and $f$ satisfy Assumption~\ref{aspt} so Theorem~\ref{th:fw} holds. The solution set is $\argmin_\mathcal{C}f=\bigcap_{k\in\mathbb{N}}\mathcal{P}_k=(1/5)\mathcal{P}_0$.
The construction is such that for all $k\in\mathbb{N}$, the points $A_k,B_{k+1},(-1,1)$, the points $B_k,C_{k+1},(1,1)$, the points $C_k,D_{k+1},(1,-1)$, and the points $D_k$, $A_{k+1}$, $(-1,-1)$, are aligned.
An illustration is presented in Figure~\ref{fig:ctr3}.

\begin{figure}[h]
\vspace{2mm}
 \centering
\begin{tikzpicture}[thick, scale=3]
  
  \coordinate (a) at (-1, -1);
  \coordinate (b) at (-1, 1);
  \coordinate (c) at (1, 1);
  \coordinate (d) at (1, -1);
  
  \coordinate (dg1) at (-1, 0.1);
  \coordinate (dg2) at (0.1, 1);
  \coordinate (dg3) at (1, -0.1);
  \coordinate (dg4) at (-0.1, -1);
  
  \coordinate (e) at (-0.02, -0.2);
  \coordinate (f) at (-0.2, 0.02);
  \coordinate (g) at (0.02, 0.2);
  \coordinate (h) at (0.2, -0.02);
  
  \coordinate (a1) at (-0.058, -0.576);
  \coordinate (b1) at (-0.576, 0.058);
  \coordinate (c1) at (0.058, 0.576);
  \coordinate (d1) at (0.576, -0.058);
  
  \coordinate (a2) at (-0.043, -0.428);
  \coordinate (b2) at (-0.428, 0.043);
  \coordinate (c2) at (0.043, 0.428);
  \coordinate (d2) at (0.428, -0.043);
  
  \coordinate (a3) at (-0.035, -0.353);
  \coordinate (b3) at (-0.353, 0.035);
  \coordinate (c3) at (0.035, 0.353);
  \coordinate (d3) at (0.353, -0.035);
  
  \coordinate (a4) at (-0.031, -0.309);
  \coordinate (b4) at (-0.309, 0.031);
  \coordinate (c4) at (0.031, 0.309);
  \coordinate (d4) at (0.309, -0.031);
  
  \coordinate (a5) at (-0.028, -0.28);
  \coordinate (b5) at (-0.28, 0.028);
  \coordinate (c5) at (0.028, 0.28);
  \coordinate (d5) at (0.28, -0.028);
  
  \coordinate (a6) at (-0.026, -0.261);
  \coordinate (b6) at (-0.261, 0.026);
  \coordinate (c6) at (0.026, 0.261);
  \coordinate (d6) at (0.261, -0.026);
  
  \coordinate (x0) at (-0.1, -1);
  \coordinate (x1) at (-0.576, 0.058);
  \coordinate (x2) at (0.043, 0.428);
  \coordinate (x3) at (0.353, -0.035);
  \coordinate (x4) at (-0.031, -0.309);
  \coordinate (x5) at (-0.28, 0.028);
  \coordinate (x6) at (0.026, 0.261);
  
  \draw[blue!40, dotted] (x0) -- (-1, 1);
  \draw[blue!40, dotted] (dg1) -- (1, 1);
  \draw[blue!40, dotted] (dg2) -- (1, -1);
  \draw[blue!40, dotted] (dg3) -- (-1, -1);
  
  \draw[blue!20, dotted] (a1) -- (-1, 1);
  \draw[blue!20, dotted] (b1) -- (1, 1);
  \draw[blue!20, dotted] (c1) -- (1, -1);
  \draw[blue!20, dotted] (d1) -- (-1, -1);
  
  \draw[orange] (dg1) -- (dg2) -- (dg3) -- (dg4) -- (dg1) -- (dg2);
  \draw[orange] (a1) -- (b1) -- (c1) -- (d1) -- (a1) -- (b1);
  \draw[orange] (a2) -- (b2) -- (c2) -- (d2) -- (a2) -- (b2);
  \draw[orange] (a3) -- (b3) -- (c3) -- (d3) -- (a3) -- (b3);
  \draw[orange] (a4) -- (b4) -- (c4) -- (d4) -- (a4) -- (b4);
  \draw[orange] (a5) -- (b5) -- (c5) -- (d5) -- (a5) -- (b5);
  \draw[orange] (a6) -- (b6) -- (c6) -- (d6) -- (a6) -- (b6);
  
  \draw (a) -- (b) -- (c) -- (d) -- (a) -- (b);

  \draw[fill=black] (dg1) circle (0.3pt) node [left] {$B_0$};
  \draw[fill=black] (dg2) circle (0.3pt) node [above] {$C_0$};
  \draw[fill=black] (dg3) circle (0.3pt) node [right] {$D_0$};
  
  \draw[fill=black] (x0) circle (0.3pt) node [below] {$A_0$};
  
  \draw[fill=black] (a1) circle (0.3pt) node [below] {$A_1$};
  \draw[fill=black] (x1) circle (0.3pt) node [left] {$B_1$};
  \draw[fill=black] (c1) circle (0.3pt) node [above] {$C_1$};
  \draw[fill=black] (d1) circle (0.3pt) node [right] {$D_1$};

  \draw[fill=black] (a) circle (0.3pt) node [below] {$(-1,-1)$};
  \draw[fill=black] (b) circle (0.3pt) node [above] {$(-1,1)$};
  \draw[fill=black] (c) circle (0.3pt) node [above] {$(1,1)$};
  \draw[fill=black] (d) circle (0.3pt) node [below] {$(1,-1)$};
  
  \draw[fill=black] (a2) circle (0.3pt);
  \draw[fill=black] (b2) circle (0.3pt);
  \draw[fill=black] (c2) circle (0.3pt);
  \draw[fill=black] (d2) circle (0.3pt);
  
  \draw[crimson, fill=crimson] (e) -- (f) -- (g) -- (h) -- (e) -- (f);

 \end{tikzpicture}
 \caption{The constraint set (in black), polygonal sketches of the objective function (in orange), and the solution set (in red). 
 For all $k\in\mathbb{N}$, the points $A_k,B_{k+1},(-1,1)$, the points $B_k,C_{k+1},(1,1)$, the points $C_k,D_{k+1},(1,-1)$, and the points $D_k,A_{k+1},(-1,-1)$, are aligned.}
 \label{fig:ctr3}
\end{figure}
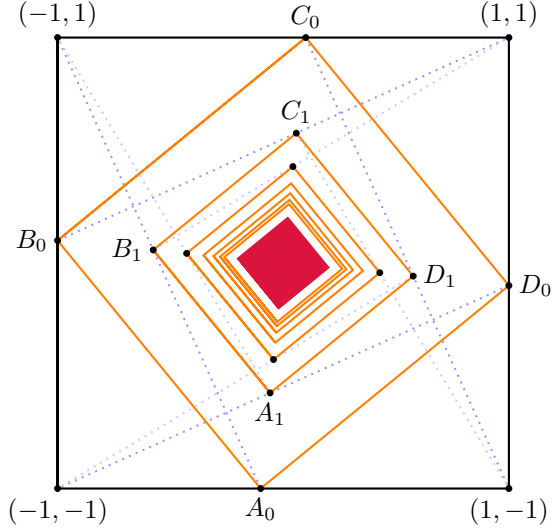

The construction allows to choose $(u_k)_{k\in\mathbb{N}}$ such that $u_0\in\mathcal{K}_{\mathcal{P}_0}A_0$ 
and, for all $k\geq1$,
\begin{align*}
 \begin{cases}
  u_k\in\mathcal{K}_{\mathcal{P}_k}A_k\text{ and }\langle A_k-D_{k-1},u_k\rangle=0&\text{if }k\equiv0\;(\operatorname{mod} 4);\\
  u_k\in\mathcal{K}_{\mathcal{P}_k}B_k\text{ and }\langle B_k-A_{k-1},u_k\rangle=0&\text{if }k\equiv1\;(\operatorname{mod} 4);\\
  u_k\in\mathcal{K}_{\mathcal{P}_k}C_k\text{ and }\langle C_k-B_{k-1},u_k\rangle=0&\text{if }k\equiv2\;(\operatorname{mod} 4);\\
  u_k\in\mathcal{K}_{\mathcal{P}_k}D_k\text{ and }\langle D_k-C_{k-1},u_k\rangle=0&\text{if }k\equiv3\;(\operatorname{mod} 4).
 \end{cases}
\end{align*}
Let $x_0=A_0$. By induction and using Lemma~\ref{cor:ls}, we obtain for all $t\in\mathbb{N}$,
\begin{align*}
 x_t=
 \begin{cases}
  A_t&\text{if }t\equiv0\;(\operatorname{mod} 4);\\
  B_t&\text{if }t\equiv1\;(\operatorname{mod} 4);\\
  C_t&\text{if }t\equiv2\;(\operatorname{mod} 4);\\
  D_t&\text{if }t\equiv3\;(\operatorname{mod} 4).
 \end{cases}
\end{align*}
Therefore, $\|x_{t+1}-x_t\|\geq\sqrt{2}/5$ for all $t\in\mathbb{N}$, where $\sqrt{2}/5$ is the side length of the square $\argmin_\mathcal{C}f=(1/5)\mathcal{P}_0$, so $(x_t)_{t\in\mathbb{N}}$ does not converge.
An illustration is presented in Figure~\ref{fig:int}.

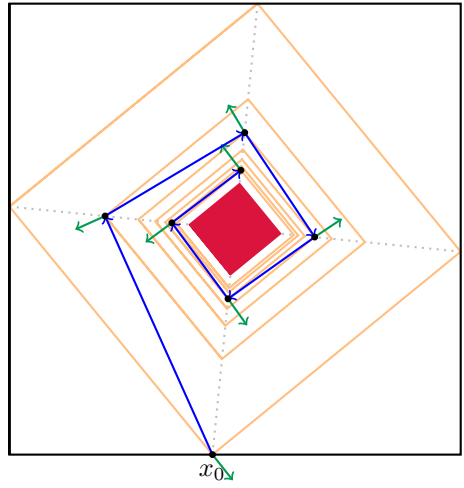
\begin{figure}[h]
\vspace{2mm}
 \centering
\begin{tikzpicture}[thick, scale=3]
  
  \coordinate (a) at (-1, -1);
  \coordinate (b) at (-1, 1);
  \coordinate (c) at (1, 1);
  \coordinate (d) at (1, -1);
  
  \coordinate (dg1) at (-1, 0.1);
  \coordinate (dg2) at (0.1, 1);
  \coordinate (dg3) at (1, -0.1);
  \coordinate (dg4) at (-0.1, -1);
  
  \coordinate (e) at (-0.02, -0.2);
  \coordinate (f) at (-0.2, 0.02);
  \coordinate (g) at (0.02, 0.2);
  \coordinate (h) at (0.2, -0.02);
  
  \coordinate (a1) at (-0.058, -0.576);
  \coordinate (b1) at (-0.576, 0.058);
  \coordinate (c1) at (0.058, 0.576);
  \coordinate (d1) at (0.576, -0.058);
  
  \coordinate (a2) at (-0.043, -0.428);
  \coordinate (b2) at (-0.428, 0.043);
  \coordinate (c2) at (0.043, 0.428);
  \coordinate (d2) at (0.428, -0.043);
  
  \coordinate (a3) at (-0.035, -0.353);
  \coordinate (b3) at (-0.353, 0.035);
  \coordinate (c3) at (0.035, 0.353);
  \coordinate (d3) at (0.353, -0.035);
  
  \coordinate (a4) at (-0.031, -0.309);
  \coordinate (b4) at (-0.309, 0.031);
  \coordinate (c4) at (0.031, 0.309);
  \coordinate (d4) at (0.309, -0.031);
  
  \coordinate (a5) at (-0.028, -0.28);
  \coordinate (b5) at (-0.28, 0.028);
  \coordinate (c5) at (0.028, 0.28);
  \coordinate (d5) at (0.28, -0.028);
  
  \coordinate (a6) at (-0.026, -0.261);
  \coordinate (b6) at (-0.261, 0.026);
  \coordinate (c6) at (0.026, 0.261);
  \coordinate (d6) at (0.261, -0.026);
  
  \coordinate (x0) at (-0.1, -1);
  \coordinate (x1) at (-0.576, 0.058);
  \coordinate (x2) at (0.043, 0.428);
  \coordinate (x3) at (0.353, -0.035);
  \coordinate (x4) at (-0.031, -0.309);
  \coordinate (x5) at (-0.28, 0.028);
  \coordinate (x6) at (0.026, 0.261);

  \draw[gray!50, dotted] (dg1) -- (dg3);
  \draw[gray!50, dotted] (dg2) -- (dg4);
  
  \draw[orange!50] (dg1) -- (dg2) -- (dg3) -- (dg4) -- (dg1) -- (dg2);
  \draw[orange!50] (a1) -- (b1) -- (c1) -- (d1) -- (a1) -- (b1);
  \draw[orange!50] (a2) -- (b2) -- (c2) -- (d2) -- (a2) -- (b2);
  \draw[orange!50] (a3) -- (b3) -- (c3) -- (d3) -- (a3) -- (b3);
  \draw[orange!50] (a4) -- (b4) -- (c4) -- (d4) -- (a4) -- (b4);
  \draw[orange!50] (a5) -- (b5) -- (c5) -- (d5) -- (a5) -- (b5);
  \draw[orange!50] (a6) -- (b6) -- (c6) -- (d6) -- (a6) -- (b6);
  
  \draw (a) -- (b) -- (c) -- (d) -- (a) -- (b);

  \draw[->][ForestGreen] (x0) -- (-0.01, -1.1136);
  \draw[->][ForestGreen] (x1) -- (-0.7081, -0.0014);
  \draw[->][ForestGreen] (x2) -- (-0.0313, 0.5524);
  \draw[->][ForestGreen] (x3) -- (0.4734, 0.0456);
  \draw[->][ForestGreen] (x4) -- (0.0532, -0.4270);
  \draw[->][ForestGreen] (x5) -- (-0.3965, -0.0581);
  \draw[->][ForestGreen] (x6) -- (-0.0618, 0.3763);

  \draw[->][blue] (x0) -- (x1);
  \draw[->][blue] (x1) -- (x2);
  \draw[->][blue] (x2) -- (x3);
  \draw[->][blue] (x3) -- (x4);
  \draw[->][blue] (x4) -- (x5);
  \draw[->][blue] (x5) -- (x6);
  
  \draw[fill=black] (x0) circle (0.3pt) node [below] {$x_0$};
  \draw[fill=black] (x1) circle (0.3pt);
  \draw[fill=black] (x2) circle (0.3pt);
  \draw[fill=black] (x3) circle (0.3pt);
  \draw[fill=black] (x4) circle (0.3pt);
  \draw[fill=black] (x5) circle (0.3pt);
  \draw[fill=black] (x6) circle (0.3pt);
  
  \draw[crimson, fill=crimson] (e) -- (f) -- (g) -- (h) -- (e) -- (f);
  
 \end{tikzpicture}
 \caption{The constraint set (in black), polygonal sketches of the objective function (in orange), the solution set (in red), gradient directions (in green), and the trajectory (in blue) of the sequence $(x_t)_{t\in\mathbb{N}}$ generated by the Frank-Wolfe algorithm with the line-search strategy~\ref{step:ls} starting from $x_0$. The sequence $(x_t)_{t\in\mathbb{N}}$ circles around the solution set and does not converge.}
 \label{fig:int}
\end{figure}

\subsection{Counterexample 3: Closed-loop strategy}
\label{sec:notls}

Our next counterexample involves the closed-loop strategy~\ref{step:closed}.

\begin{ctr}
 There exist $\mathcal{C}$ and $f$, satisfying Assumption~\ref{aspt}, and $x_0\in\mathcal{C}$ such that any sequence $(x_t)_{t\in\mathbb{N}}$ generated by Frank-Wolfe algorithm using the closed-loop strategy~\ref{step:closed} does not converge.
\end{ctr}

Let $\mathcal{C}=[-1,1]\times[0,2^K]$, where $K\in\mathbb{N}$, and $f$ be defined via Theorem~\ref{th:counter} with, for all $k\in\mathbb{N}$, 
\begin{alignat*}{2}
 &A_k=\left((-1)^{k+1}\frac{61}{35},\frac{9}{8}\frac{61}{35}\frac{1}{2^{k+1-K}}\right),
 &&B_k=\left((-1)^k\frac{61}{35},\frac{17}{16}\frac{61}{35}\frac{1}{2^{k+1-K}}\right),\\
 &D_k=\left((-1)^{k+1},\frac{9}{8}\frac{1}{2^{k+1-K}}\right),
 &&C_k=\left((-1)^k,\frac{17}{16}\frac{1}{2^{k+1-K}}\right),\\
 &D_k'=\left((-1)^{k+1},\frac{1}{2^{k+1-K}}\right),
 &&C_k'=\left((-1)^k\frac{61}{35},\frac{1}{2^{k+1-K}}\right),\\
 &Y_k=
\begin{cases}
  \left(-\displaystyle\frac{61}{35}\left(1+\displaystyle\frac{1}{2^k}\right),0\right)&\text{if }k\text{ is even};\\
  \left(-\displaystyle\frac{61}{35}\left(1+\displaystyle\frac{17}{16}\frac{8}{9}\frac{1}{2^k}\right),0\right)&\text{if }k\text{ is odd},
 \end{cases}
 \quad\quad
 &&Y_k'=\left(-\frac{61}{35}-\frac{8}{9}\frac{1}{2^k},0\right),
 \\
 &Z_k=
\begin{cases}
  \left(\displaystyle\frac{61}{35}\left(1+\displaystyle\frac{1}{2^k}\right),0\right)&\text{if }k\text{ is even};\\
  \left(\displaystyle\frac{61}{35}\left(1+\displaystyle\frac{9}{8}\frac{16}{17}\frac{1}{2^k}\right),0\right)&\text{if }k\text{ is odd},
 \end{cases}
 &&Z_k'=\left(\frac{61}{35}+\frac{16}{17}\frac{1}{2^k},0\right),
\end{alignat*}
and
$\mathcal{P}_{2k}=\operatorname{conv}\{Y_k,A_k,B_k,Z_k,X_{2k}\}$ and $\mathcal{P}_{2k+1}=\operatorname{conv}\{Y_k',D_k',D_k,C_k,C_k',Z_k',X_{2k+1}\}$, where $X_k=(0,-1-1/(k+1))$.
Then $\mathcal{C}$ and $f$ satisfy Assumption~\ref{aspt} so Theorem~\ref{th:fw} holds. The solution set is $\argmin_\mathcal{C}f=\bigcap_{k\in\mathbb{N}}\mathcal{P}_k\cap\mathcal{C}=\left[(-1,0),(1,0)\right]$.
The construction is such that for all $k\in\mathbb{N}$, the lines $(A_kB_k)$ and $(C_kD_k)$ are parallel, $\langle B_k,e_2\rangle<\langle A_k,e_2\rangle$ and $\langle C_k,e_2\rangle<\langle D_k,e_2\rangle$, and the points $0$, $D_k$, $A_k$, and the points $0$, $C_k$, $B_k$, are aligned.
An illustration is presented in Figure~\ref{fig:all}.

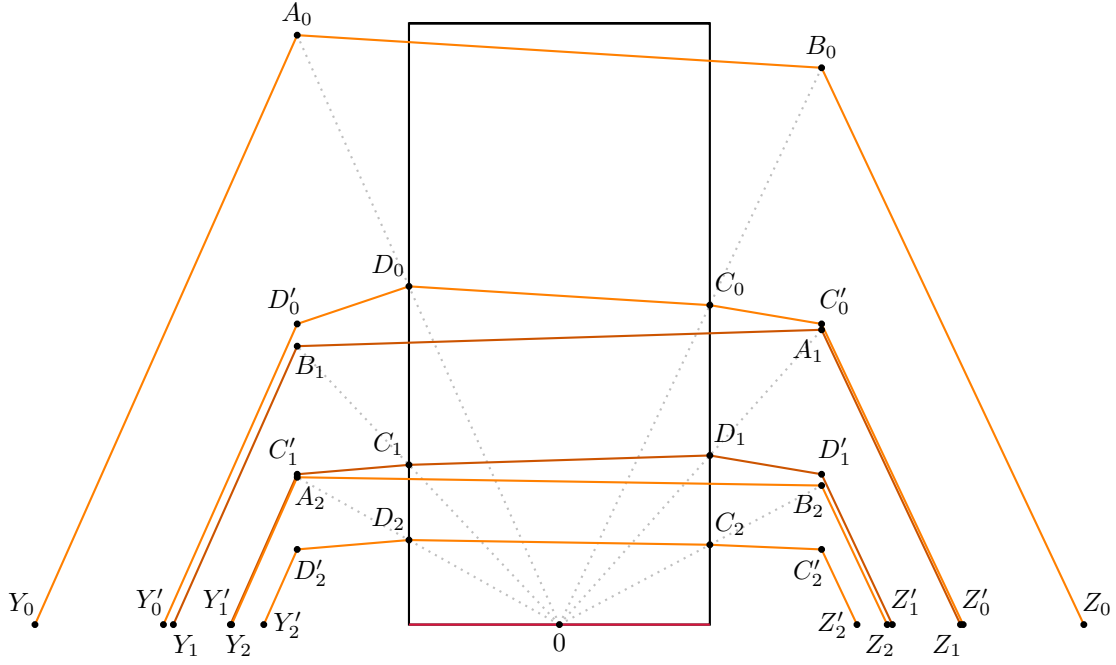
\begin{figure}[h]
\vspace{2mm}
 \centering
\begin{tikzpicture}[thick, scale=2]

  \coordinate (o) at (0, 0);
  
  \coordinate (s1) at (-2, 4);
  \coordinate (s2) at (2, 4);
  \coordinate (s3) at (-2, 2);
  \coordinate (s4) at (2, 2);
  \coordinate (s5) at (-2, 1);
  \coordinate (s6) at (2, 1);
  \coordinate (s7) at (-2, 1/2);
  \coordinate (s8) at (2, 1/2);
  
  \coordinate (S1) at (-1/2, 4);
  \coordinate (S2) at (1/2, 4);
  \coordinate (S3) at (1/2, 0);
  \coordinate (S4) at (-1/2, 0);
  
  \coordinate (C1) at (-1, 4);
  \coordinate (C2) at (1, 4);
  \coordinate (C3) at (1, 0);
  \coordinate (C4) at (-1, 0);
  
  \coordinate (a) at (-61/35, 3.921);
  \coordinate (b) at (61/35, 3.704);
  \coordinate (c) at (1, 2.125);
  \coordinate (d) at (-1, 2.25);
  \coordinate (c') at (61/35, 2);
  \coordinate (d') at (-61/35, 2);
  \coordinate (y) at (-3.486, 0);
  \coordinate (z) at (3.486, 0);
  \coordinate (y') at (-2.632, 0);
  \coordinate (z') at (2.684, 0);
  \coordinate (na) at (-1.718, 4.101);
  \coordinate (nb) at (1.768, 3.884);
  \coordinate (nc) at (1.025, 2.305);
  \coordinate (nd) at (-0.975, 2.43);

  \coordinate (aa) at (61/35, 1.961);
  \coordinate (bb) at (-61/35, 1.852);
  \coordinate (cc) at (-1, 1.0625);
  \coordinate (dd) at (1, 1.125);
  \coordinate (cc') at (-61/35, 1);
  \coordinate (dd') at (61/35, 1);
  \coordinate (yy) at (-2.566, 0);
  \coordinate (zz) at (2.666, 0);
  \coordinate (yy') at (-2.187, 0);
  \coordinate (zz') at (2.213, 0);

  \coordinate (aaa) at (-61/35, 0.980);
  \coordinate (bbb) at (61/35, 0.925);
  \coordinate (ccc) at (1, 0.531);
  \coordinate (ddd) at (-1, 0.5625);
  \coordinate (ccc') at (61/35, 1/2);
  \coordinate (ddd') at (-61/35, 1/2);
  \coordinate (yyy) at (-2.179, 0);
  \coordinate (zzz) at (2.179, 0);
  \coordinate (yyy') at (-1.965, 0);
  \coordinate (zzz') at (1.978, 0);
  
  \coordinate (e) at (-1, 7/2);
  \coordinate (f) at (1, 7/2);
  \coordinate (g) at (1, 5/2);
  \coordinate (h) at (-1, 5/2);
  
  \coordinate (i) at (-1, 13/4);
  \coordinate (j) at (1, 13/4);

 \draw[gray!50, dotted] (o) -- (bbb);
 \draw[gray!50, dotted] (o) -- (aaa);
 \draw[gray!50, dotted] (o) -- (bb);
 \draw[gray!50, dotted] (o) -- (aa);
 \draw[gray!50, dotted] (o) -- (a);
 \draw[gray!50, dotted] (o) -- (b);
 
 \draw (C1) -- (C2) -- (C3) -- (C4) -- (C1) -- (C2);
\draw[crimson] (C4) -- (C3);

 \draw[orange] (y) -- (a) -- (b) -- (z);
  
 \draw[orange] (y') -- (d') -- (d) -- (c) -- (c') -- (z');
 
 \draw[burntorange] (yy) -- (bb) -- (aa) -- (zz);
 \draw[burntorange] (yy') -- (cc') -- (cc) -- (dd) -- (dd') -- (zz');
 
 \draw[orange] (yyy) -- (aaa) -- (bbb) -- (zzz);
 \draw[orange] (yyy') -- (ddd') -- (ddd) -- (ccc) -- (ccc') -- (zzz');

  \draw[fill=black] (o) circle (0.45pt) node [below] {$0$};
  
  \draw[fill=black] (a) circle (0.45pt) node [above] {$A_0$};
  \draw[fill=black] (b) circle (0.45pt) node [above] {$B_0$};
  \draw[fill=black] (c) circle (0.45pt) node [above] {$\quad\;\;C_0$};
  \draw[fill=black] (d) circle (0.45pt) node [above] {$D_0\quad\;\;$};
  \draw[fill=black] (c') circle (0.45pt) node [above] {$\quad C_0'$};
  \draw[fill=black] (d') circle (0.45pt) node [above] {$D_0'\quad$};
  \draw[fill=black] (y) circle (0.45pt) node [above] {$Y_0\quad$};
  \draw[fill=black] (y') circle (0.45pt) node [above] {$Y_0'\quad$};
  \draw[fill=black] (z) circle (0.45pt) node [above] {$\quad Z_0$};
  \draw[fill=black] (z') circle (0.45pt) node [above] {$\quad Z_0'$};

\draw[fill=black] (aa) circle (0.45pt) node [below] {$A_1\quad$};
\draw[fill=black] (bb) circle (0.45pt) node [below] {$\quad B_1$};
\draw[fill=black] (cc) circle (0.45pt) node [above] {$C_1\quad\;\;$};
\draw[fill=black] (dd) circle (0.45pt) node [above] {$\quad\;\;D_1$};
\draw[fill=black] (cc') circle (0.45pt) node [above] {$C_1'\quad$};
\draw[fill=black] (dd') circle (0.45pt) node [above] {$\quad D_1'$};
\draw[fill=black] (yy) circle (0.45pt) node [below] {$\quad Y_1$};
\draw[fill=black] (yy') circle (0.45pt) node [above] {$Y_1'\quad$};
\draw[fill=black] (zz) circle (0.45pt) node [below] {$Z_1\quad$};
\draw[fill=black] (zz') circle (0.45pt) node [above] {$\quad Z_1'$};

\draw[fill=black] (aaa) circle (0.45pt) node [below] {$\quad A_2$};
\draw[fill=black] (bbb) circle (0.45pt) node [below] {$B_2\quad$};
\draw[fill=black] (ccc) circle (0.45pt) node [above] {$\quad\;\;C_2$};
\draw[fill=black] (ddd) circle (0.45pt) node [above] {$D_2\quad\;\;$};
\draw[fill=black] (ccc') circle (0.45pt) node [below] {$C_2'\quad$};
\draw[fill=black] (ddd') circle (0.45pt) node [below] {$\quad D_2'$};
\draw[fill=black] (yyy) circle (0.45pt) node [below] {$\;\;Y_2$};
\draw[fill=black] (zzz) circle (0.45pt) node [below] {$Z_2\;\;$};
\draw[fill=black] (yyy') circle (0.45pt) node [right] {$Y_2'$};
\draw[fill=black] (zzz') circle (0.45pt) node [left] {$Z_2'$};
 \end{tikzpicture}
 \caption{The constraint set (in black), polygonal sketches of the objective function (in orange and dark orange), and the solution set (in red). The polygonal sketches are studied by pairs, corresponding to the color tone, i.e., to the direction in which they lean (see Figure~\ref{fig:all2}). For all $k\in\mathbb{N}$, the lines $(A_kB_k)$ and $(C_kD_k)$ are parallel, $\langle B_k,e_2\rangle\leq\langle A_k,e_2\rangle$ and $\langle C_k,e_2\rangle\leq\langle D_k,e_2\rangle$, and the points $0$, $D_k$, $A_k$, and the points $0$, $C_k$, $B_k$, are aligned.}
 \label{fig:all}
\end{figure}

By properties of the construction, for all $k\in\mathbb{N}$,
\begin{alignat*}{2}
 &\mathcal{K}_{\mathcal{P}_{2k+1}}D_k
 \subset\mathcal{K}_{\mathcal{P}_{2k}}A_k,
 &&\mathcal{K}_{\mathcal{P}_{2k+1}}C_k
 \subset\mathcal{K}_{\mathcal{P}_{2k}}B_k,\\
 &\mathcal{K}_{\mathcal{P}_{4k+1}}D_{2k}
 \cap(\mathbb{R}^*_+)^2\neq\varnothing,
  &&\mathcal{K}_{\mathcal{P}_{4k+1}}C_{2k}
  \cap(\mathbb{R}^*_+)^2\neq\varnothing,\\
 &\mathcal{K}_{\mathcal{P}_{4k+3}}D_{2k+1}
 \cap(\mathbb{R}^*_-\times\mathbb{R}^*_+)\neq\varnothing,
\quad\quad&&\mathcal{K}_{\mathcal{P}_{4k+3}}C_{2k+1}
\cap(\mathbb{R}^*_-\times\mathbb{R}^*_+)\neq\varnothing.
\end{alignat*}
Thus, we can choose $(u_k)_{k\in\mathbb{N}}$ such that for all $k\in\mathbb{N}$,
\begin{enumerate}[label=(\roman*)]
 \item\label{u1} $\nabla f(A_k)$ and $\nabla f(D_k)$, and $\nabla f(B_k)$ and $\nabla f(C_k)$, are positively colinear;
 \item\label{u2} $\langle e_2,\nabla f(x)\rangle>0$ and $(-1)^k\langle e_1,\nabla f(x)\rangle>0$ for all $x\in\{A_k,B_k,C_k,D_k\}$.
\end{enumerate}
Then, for all $k\in\mathbb{N}$, let
\begin{alignat*}{2}
 &E_k=\left((-1)^{k+1},\frac{7}{4}\frac{1}{2^{k+1-K}}\right),\quad\quad
 &&F_k=\left((-1)^k,\frac{7}{4}\frac{1}{2^{k+1-K}}\right),\\
 &I_k=\left((-1)^{k+1},\frac{13}{8}\frac{1}{2^{k+1-K}}\right),
 &&J_k=\left((-1)^k,\frac{13}{8}\frac{1}{2^{k+1-K}}\right),\\
 &H_k=\left((-1)^{k+1},\frac{5}{4}\frac{1}{2^{k+1-K}}\right),\quad
 &&G_k=\left((-1)^k,\frac{5}{4}\frac{1}{2^{k+1-K}}\right).
\end{alignat*}
An illustration is presented in Figure~\ref{fig:all2}.

\begin{figure}[h]
\vspace{2mm}
 \centering
\begin{tikzpicture}[thick, scale=2]

  \coordinate (o) at (0, 0);
  
  \coordinate (s1) at (-2, 4);
  \coordinate (s2) at (2, 4);
  \coordinate (s3) at (-2, 2);
  \coordinate (s4) at (2, 2);
  \coordinate (s5) at (-2, 1);
  \coordinate (s6) at (2, 1);
  \coordinate (s7) at (-2, 1/2);
  \coordinate (s8) at (2, 1/2);
  
  \coordinate (S1) at (-1/2, 4.021);
  \coordinate (S2) at (1/2, 4.021);
  \coordinate (S3) at (1/2, 1.9);
  \coordinate (S4) at (-1/2, 1.9);
  
  \coordinate (C1) at (-1, 4.021);
  \coordinate (C2) at (1, 4.021);
  \coordinate (C3) at (1, 1.9);
  \coordinate (C4) at (-1, 1.9);
  
  \coordinate (a) at (-61/35, 3.921);
  \coordinate (b) at (61/35, 3.704);
  \coordinate (c) at (1, 2.125);
  \coordinate (d) at (-1, 2.25);
  \coordinate (c') at (61/35, 2);
  \coordinate (d') at (-61/35, 2);
  \coordinate (y) at (-3.485, 0);
  \coordinate (z) at (3.485, 0);
  \coordinate (y') at (-2.632, 0);
  \coordinate (z') at (2.684, 0);
\coordinate (na) at (-1.7328, 4.1207);
  \coordinate (nb) at (1.771, 3.902);
  \coordinate (nc) at (1.028, 2.323);
\coordinate (nd) at (-0.99, 2.4497);
  
  \coordinate (e) at (-1, 7/2);
  \coordinate (f) at (1, 7/2);
  \coordinate (g) at (1, 5/2);
  \coordinate (h) at (-1, 5/2);
  
  \coordinate (i) at (-1, 13/4);
  \coordinate (j) at (1, 13/4);
  \coordinate (k) at (1, 11/4);
  \coordinate (l) at (-1, 11/4);
  
  \draw[gray!67, fill=gray!67] (e) -- (f) -- (j) -- (i);
  \draw[gray!33, fill=gray!33] (i) -- (j) -- (g) -- (h); 
  
 \draw[dotted] (S1) -- (S4);
 \draw[dotted] (S2) -- (S3);
 \draw (C1) -- (C4);
 \draw (C2) -- (C3);
 
 \draw[orange] (-2.597, 2) -- (a) -- (b) -- (2.544, 2);
 
  \draw[->][ForestGreen] (a) -- (na);
  \draw[->][ForestGreen] (b) -- (nb);
  \draw[->][ForestGreen] (c) -- (nc);
  \draw[->][ForestGreen] (d) -- (nd);
  
 \draw[orange] (d') -- (d) -- (c) -- (c');
  
  \draw[fill=black] (a) circle (0.45pt) node [left] {$A_k$};
  \draw[fill=black] (b) circle (0.45pt) node [right] {$B_k$};
  \draw[fill=black] (c) circle (0.45pt) node [right] {$C_k$};
  \draw[fill=black] (d) circle (0.45pt) node [left] {$D_k$};
  \draw[fill=black] (c') circle (0.45pt) node [right] {$C_k'$};
  \draw[fill=black] (d') circle (0.45pt) node [left] {$D_k'$};
  
  \draw[fill=black] (e) circle (0.45pt) node [left] {$E_k$};
  \draw[fill=black] (f) circle (0.45pt) node [right] {$F_k$}; 
  \draw[fill=black] (g) circle (0.45pt) node [right] {$G_k$};
  \draw[fill=black] (h) circle (0.45pt) node [left] {$H_k$};
  
  \draw[fill=black] (i) circle (0.45pt) node [left] {$I_k$};
  \draw[fill=black] (j) circle (0.45pt) node [right] {$J_k$};
  
 \end{tikzpicture}
 \caption{A pair of polygonal sketches (in orange) and gradient directions (in green). Here, $k\in\mathbb{N}$ is even. Using the gradient directions, we show that the linear minimization oracle always returns $(-1, 0)$ in $\operatorname{conv}\{E_k,F_k,G_k,H_k\}$ here (Lemma~\ref{prop}\ref{iv}). Then, we show that, if $k$ is large enough, there exists an iterate $x_t\in\operatorname{conv}\{E_k,F_k,J_k,I_k\}$ (Lemma~\ref{prop}\ref{prop:k0}). Put together and using Thales' theorem, we can measure the minimum horizontal displacement of the iterates when they cross (vertically) the rectangle $\operatorname{conv}\{I_k,J_k,G_k,H_k\}$. This turns out to be a constant number, showing that $(x_t)_{t\in\mathbb{N}}$ does not converge (Lemma~\ref{prop}\ref{prop:end}).}
 \label{fig:all2}
\end{figure}
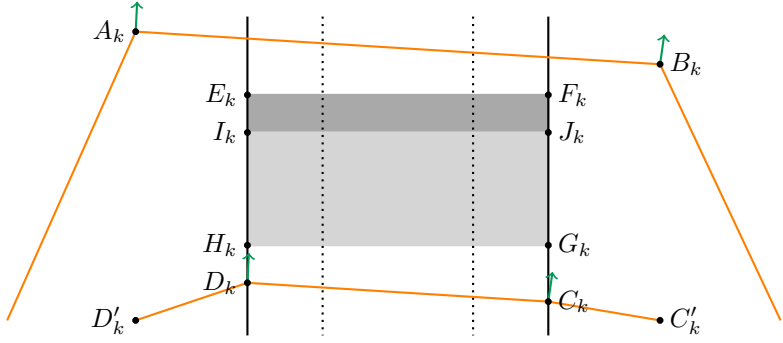

\newpage
\begin{lemma}
 \label{lem:efgh}
 Let $k\in\mathbb{N}$ and $x\in\mathcal{C}$. 
 \begin{enumerate}[label=(\roman*)]
  \item\label{e2} If $\langle x,e_2\rangle>0$, then $\langle e_2,\nabla f(x)\rangle>0$.
  \item\label{e1} If $x\in\operatorname{conv}\{E_k,F_k,G_k,H_k\}$, then $(-1)^k\langle e_1,\nabla f(x)\rangle>0$.
 \end{enumerate}
\end{lemma}

\begin{proof}
\begin{enumerate}[label=(\roman*)]
 \item Let $x^*
=x-\langle x,e_2\rangle e_2
\in\left[(-1,0),(1,0)\right]=\argmin_\mathcal{C}f$. By convexity of $f$, $\langle x^*-x,\nabla f(x)\rangle<0$, i.e., $\langle x,e_2\rangle\langle e_2,\nabla f(x)\rangle>0$. Since $\langle x,e_2\rangle>0$, we obtain $\langle e_2,\nabla f(x)\rangle>0$.
\item By Lemma~\ref{lem:align} and items~\ref{u1}--\ref{u2} above, $(-1)^k\langle e_1,\nabla f(y)\rangle>0$ for all $y\in\left[A_k,D_k\right]\cup\left[B_k,C_k\right]$.
There exist $y_1\in\left[A_k,D_k\right]$,  $y_2\in\left[B_k,C_k\right]$, and $\gamma'\in\left[0,1\right]$ such that 
$y_2-y_1=\|y_2-y_1\|e_1$ and $x=(1-\gamma')y_1+\gamma'y_2$.
Let
$\varphi\colon\gamma\in\left[0,1\right]\mapsto f((1-\gamma)y_1+\gamma y_2)$, which is convex. Then $\varphi'(\gamma)=\langle y_2-y_1,\nabla f((1-\gamma)y_1+\gamma y_2)\rangle=\|y_2-y_1\|\langle e_1,\nabla f((1-\gamma)y_1+\gamma y_2)\rangle$ for all $\gamma\in\left[0,1\right]$. 
Thus, $(-1)^k\varphi'(0)=(-1)^k\|y_2-y_1\|\langle e_1,\nabla f(y_1)\rangle>0$ and $(-1)^k\varphi'(1)=(-1)^k\|y_2-y_1\|\langle e_1,\nabla f(y_2)\rangle>0$, so, by monotonicity of $\varphi'$, $(-1)^k\varphi'(\gamma)>0$ for all $\gamma\in\left[0,1\right]$. Therefore, $(-1)^k\langle e_1,\nabla f(x)\rangle>0$.
\end{enumerate}
\end{proof}

Lemma~\ref{prop} completes our counterexample. 
Note that $L>L_\mathcal{C}^*\stackrel{\text{def}}{=}\sup_{x,y\in\mathcal{C},x\neq y}\|\nabla f(y)-\nabla f(x)\|/\|y-x\|$ is most likely in practice, because the optimal value $L_\mathcal{C}^*$ is very hard to estimate precisely in general. 

\begin{lemma}
\label{prop}
 Consider the Frank-Wolfe algorithm 
 starting from $x_0\in\mathcal{C}\setminus\argmin_\mathcal{C}f$ with the closed-loop strategy~\ref{step:closed}. If $L>L_\mathcal{C}^*$, then:
 \begin{enumerate}[label=(\roman*)]
\item\label{i} For all $t\in\mathbb{N}$,  $x_t\notin\argmin_\mathcal{C}f$;
\item\label{ii} For all $t\in\mathbb{N}$, $v_t\in\left[(-1,0),(1,0)\right]$ and $\gamma_t<1$;
\item\label{iii} $\gamma_t\to0$;
  \item\label{iv} For all $t,k\in\mathbb{N}$, if $x_t\in\operatorname{conv}\{E_k,F_k,G_k,H_k\}$, then $v_t=((-1)^{k+1},0)$;
  \item\label{prop:k0} There exists $k_0\in\mathbb{N}$ such that for all $k\geq k_0$, $\{x_t\mid t\in\mathbb{N}\}\cap\operatorname{conv}\{E_k,F_k,J_k,I_k\}\neq\varnothing$;
  \item\label{prop:end} $(x_t)_{t\in\mathbb{N}}$ does not converge.
 \end{enumerate}
\end{lemma}

\begin{proof}
 Let $x_t=(\alpha_t,\beta_t)$ for all $t\in\mathbb{N}$ and recall that $\argmin_\mathcal{C}f=\left[(-1,0),(1,0)\right]$. 
 \begin{enumerate}[label=(\roman*)]
  \item We proceed by induction. The base case is satisfied. Let $t\in\mathbb{N}$ be such that $x_t\notin\argmin_\mathcal{C}f$. Then $\beta_t>0$ and, by Lemma~\ref{lem:efgh}\ref{e2}, $v_t\in\left[(-1,0),(1,0)\right]$, so $\nabla f(v_t)=0$. By the Cauchy-Schwarz inequality,
\begin{align}
  \gamma_t
  \leq\frac{\langle x_t-v_t,\nabla f(x_t)\rangle}{L\|x_t-v_t\|^2}
  \leq\frac{\|\nabla f(x_t)\|}{L\|x_t-v_t\|}
  =\frac{\|\nabla f(x_t)-\nabla f(v_t)\|}{L\|x_t-v_t\|}
  \leq\frac{L_\mathcal{C}^*}{L}
  <1.\label{g1}
\end{align}
Therefore, $x_{t+1}\in\left[x_t,v_t\right[\subset\mathcal{C}\setminus\argmin_\mathcal{C}f$.
\item See the proof of Lemma~\ref{prop}\ref{i}.
\item By Theorem~\ref{th:counter}, $f$ is $d$-time continuously differentiable, where $d\geq2$. By~\eqref{g1} and the mean value theorem, 
\begin{align}
  \gamma_t
  \leq\frac{\|\nabla f(x_t)-\nabla f(v_t)\|}{L\|x_t-v_t\|}
  \leq\frac{1}{L}\sup_{x\in\left[x_t,v_t\right]}\|\nabla^2f(x)\|_{\operatorname{op}},\label{op}
\end{align}
where $\|\cdot\|_{\operatorname{op}}$ denotes the operator norm. Since $\argmin_\mathcal{D}f=\bigcap_{k\in\mathbb{N}}\mathcal{P}_k$ has nonempty interior, $\nabla^2f(x)=0$ for all $x\in\argmin_\mathcal{D}f\supset\argmin_\mathcal{C}f$. 
By Theorem~\ref{th:fw}, $f(x_t)\to\min_\mathcal{C}f$, so $\operatorname{dist}(x_t,\argmin_\mathcal{C}f)\to0$, and, by Lemma~\ref{prop}\ref{ii}, $v_t\in\argmin_\mathcal{C}f$ for all $t\in\mathbb{N}$.
By continuity of $\nabla^2f$, it follows that $\sup_{x\in\left[x_t,v_t\right]}\|\nabla^2f(x)\|_{\operatorname{op}}\to0$. 
By~\eqref{op}, $\gamma_t\to0$. 
  \item This follows from Lemma~\ref{lem:efgh}\ref{e1}.
  \item By Lemma~\ref{prop}\ref{ii}, $v_t\in\left[(-1,0),(1,0)\right]$ for all $t\in\mathbb{N}$. Thus,
 \begin{align}
  \beta_{t'}
  =\left(\prod_{\ell=t}^{t'-1}(1-\gamma_\ell)\right)\beta_t,\label{zz}
 \end{align}
 for all $t,t'\in\mathbb{N}$ such that $t'\geq t+1$. Let $c=1-\sqrt{13/14}>0$. By Theorem~\ref{th:fw}, $f(x_t)\to\min_\mathcal{C}f$, so $\beta_t\to0$. By Lemma~\ref{prop}\ref{iii}, $\gamma_t\to0$, so there exists $t_0\in\mathbb{N}$ such that $\beta_t\leq\beta_{t_0}$ and $\gamma_t\leq c$ for all $t\geq t_0$. Let $k_0\in\mathbb{N}$ be such that $\beta_{t_0}\geq(7/4)(1/2)^{k_0+1-K}=\langle E_{k_0},e_2\rangle$. Let $k\geq k_0$, $t\geq t_0$, and $t'\geq t+1$ be such that $\beta_t\geq(7/4)(1/2)^{k+1-K}=\langle E_k,e_2\rangle$ and $\beta_{t'}\leq(13/8)(1/2)^{k+1-K}=\langle I_k,e_2\rangle$. By~\eqref{zz},
 \begin{align*}
  \frac{13/8}{7/4}
  \geq\prod_{\ell=t}^{t'-1}(1-\gamma_\ell)
  \geq(1-c)^{t'-t},
 \end{align*}
 so
 \begin{align*}
  t'-t
  \geq\frac{\ln(14/13)}{\ln(1/(1-c))}
  =2.
 \end{align*}
 Therefore, there exists $t''\in\llbracket t+1,t'-1\rrbracket$ such that $x_{t''}\in\operatorname{conv}\{E_k,F_k,J_k,I_k\}$.
 
 \item Let $t,k\in\mathbb{N}$ be such that $x_t\in\operatorname{conv}\{E_k,F_k,J_k,I_k\}$. With no loss of generality, we can assume that $k$ is even.
 \begin{itemize}
  \item If $\alpha_t\geq-1/2$, let $t'=\min\{\ell\in\mathbb{N}\mid\beta_\ell<(5/4)(1/2)^{k+1-K}=\langle H_k,e_2\rangle,\ell\geq t\}$. By Lemma~\ref{prop}\ref{iv}, $v_\ell=(-1,0)$ for every $\ell\in\llbracket t,t'-1\rrbracket$. By Thales' theorem,
 \begin{align*}
  \frac{\alpha_t-\alpha_{t'}}{\alpha_t-(-1)}
  =\frac{\beta_t-\beta_{t'}}{\beta_t-0},
 \end{align*}
 where $\alpha_t\geq-1/2$, $\beta_t\geq(13/8)(1/2)^{k+1-K}=\langle I_k,e_2\rangle$, and $\beta_{t'}\leq(5/4)(1/2)^{k+1-K}=\langle H_k,e_2\rangle$. Thus,
 \begin{align*}
  \alpha_t-\alpha_{t'}
  \geq\frac{13/8-5/4}{13/8}\frac{1}{2}
  =\frac{3}{26}.
 \end{align*}
 \item Else, $\alpha_t<-1/2$. There exists $t'\in\mathbb{N}$ such that $x_{t'}\in\operatorname{conv}\{E_{k+1},F_{k+1},J_{k+1},I_{k+1}\}$. Let $t''=\min\{\ell\in\mathbb{N}\mid\beta_\ell<(5/4)(1/2)^{k+2-K}=\langle H_{k+1},e_2\rangle,\ell\geq t'\}$. If $\alpha_{t'}\leq1/2$, we can reproduce the same argument as above and conclude that
 \begin{align*}
  \alpha_{t''}-\alpha_{t'}
  \geq\frac{3}{26}.
 \end{align*}
 Else, $\alpha_{t'}\geq1/2$, so 
 \begin{align*}
  \alpha_{t'}-\alpha_t
  \geq1
  \geq\frac{3}{26}.
 \end{align*}
 \end{itemize}
 Thus, together with Lemma~\ref{prop}\ref{prop:k0}, we conclude that $\operatorname{card}\{t\in\mathbb{N}\mid\exists t'\geq t\colon|\alpha_{t'}-\alpha_t|\geq3/26\}=+\infty$. Therefore, $(\alpha_t)_{t\in\mathbb{N}}$ is not a Cauchy sequence, so $(x_t)_{t\in\mathbb{N}}$ does not converge.
 \end{enumerate}
\end{proof}

\subsection{Counterexample 4: Open-loop strategies}

Our last counterexample involves the open-loop strategies~\ref{step:open1}--\ref{step:open2}. It is similar in design to the one in Section~\ref{sec:ls}. Note that the counterexample from Section~\ref{sec:notls} would be valid here as well if we had $\gamma_0<1$.

\begin{ctr}
 There exist $\mathcal{C}$ and $f$, satisfying Assumption~\ref{aspt}, and $x_0\in\mathcal{C}$ such that any sequence $(x_t)_{t\in\mathbb{N}}$ generated by Frank-Wolfe algorithm using an open-loop strategy~\ref{step:open1}--\ref{step:open2} does not converge.
\end{ctr}

Let $(\gamma_t)_{t\in\mathbb{N}}$ be an open-loop strategy, $\mathcal{C}=\operatorname{conv}\{(-2,1/4),(-1,0),(0,1),(1,0)\}$, and $f$ be defined via Theorem~\ref{th:counter} with, for all $k\in\mathbb{N}$, 
$\mathcal{P}_k=\operatorname{conv}\{A_k,B_k,C_k,-B_k\}$ where
\begin{align*}
 A_k=\left(-2-\displaystyle\frac{1}{k+1},0\right),
 \quad
 C_k=\left(1+\displaystyle\frac{1}{k+1},0\right),
\end{align*}
$B_0=(0,1)$, $B_1=V_0=(-2,1/4)$, $V_1=(1,0)$, and, for all $k\geq1$, 
\begin{align*}
 &B_{k+1}=(1-\gamma_k)B_k+\gamma_kV_k,\\
 &V_{k+1}=
 \begin{cases}
  -V_k&\text{if }|\langle B_{k+1},e_1\rangle|>1/4\text{ and }|\langle B_k,e_1\rangle|\leq1/4;\\
  V_k&\text{else}.
 \end{cases}
\end{align*}
Then $\mathcal{C}$ and $f$ satisfy Assumption~\ref{aspt} so Theorem~\ref{th:fw} holds. The solution set is $\argmin_\mathcal{C}f=\bigcap_{k\in\mathbb{N}}\mathcal{P}_k\cap\mathcal{C}=\left[(-1,0),(1,0)\right]$. The construction is such that the trajectory of $(B_k)_{k\in\mathbb{N}}$ is headed by $(V_k)_{k\in\mathbb{N}}$, whose terms change value each time $(B_k)_{k\in\mathbb{N}}$ crosses the vertical band $\{x\in\mathcal{C}\mid|\langle x,e_1\rangle|\leq1/4\}$. An illustration is presented in Figure~\ref{fig:open}.

\begin{figure}[!h]
 \centering 
 \begin{tikzpicture}[thick, xscale=2, yscale=12]
  
  \coordinate (A0) at (-3, 0);
  \coordinate (B0) at (0, 1);
  \coordinate (C0) at (2, 0);
  
  \coordinate (A1) at (-2.5, 0);
  \coordinate (B1) at (-2, 1/4);
  \coordinate (C1) at (1.5, 0);
  
  \coordinate (A2) at (-2.333, 0);
  \coordinate (B2) at (0, 0.0833);
  \coordinate (C2) at (1.333, 0);
  
  \coordinate (A3) at (-2.25, 0);
  \coordinate (B3) at (0.5, 0.0417);
  \coordinate (C3) at (1.25, 0);
  
  \coordinate (A4) at (-2.2, 0);
  \coordinate (B4) at (-0.1, 0.025);
  \coordinate (C4) at (1.2, 0);
  
  \coordinate (A5) at (-2.1667, 0);
  \coordinate (B5) at (-0.4, 0.0167);
  \coordinate (C5) at (1.1667, 0);
  
  \coordinate (A6) at (-2.149, 0);
  \coordinate (B6) at (0, 0.0119);
  \coordinate (C6) at (1.1429, 0);
  
  \coordinate (A7) at (-2.125, 0);
  \coordinate (B7) at (0.25, 0.0089);
  \coordinate (C7) at (1.125, 0);
  
  \coordinate (A8) at (-2.111, 0);
  \coordinate (B8) at (0.4167, 0.0069);
  \coordinate (C8) at (1.111, 0);
  
  \coordinate (A9) at (-2.1, 0);
  \coordinate (B9) at (0.1333, 0.0056);
  \coordinate (C9) at (1.1, 0);
  
  \draw[dotted] (-1/4, 0) -- (-1/4, 1);
  \draw[dotted] (1/4, 0) -- (1/4, 1);
  
  \draw[orange!50] (A0) -- (B0) -- (C0);
  \draw[orange!50] (A1) -- (B1) -- (C1);
  \draw[orange!50] (A2) -- (B2) -- (C2);
  \draw[orange!50] (A3) -- (B3) -- (C3);
  \draw[orange!50] (A4) -- (B4) -- (C4);
  \draw[orange!50] (A5) -- (B5) -- (C5);
  \draw[orange!50] (A6) -- (B6) -- (C6);
  \draw[orange!50] (A7) -- (B7) -- (C7);
  \draw[orange!50] (A8) -- (B8) -- (C8);
  \draw[orange!50] (A9) -- (B9) -- (C9);

  \draw[->][ForestGreen] (B1) -- (-2.03, 0.287);
  \draw[->][ForestGreen] (B2) -- (-0.03, 0.1203);
  \draw[->][ForestGreen] (B3) -- (0.47, 0.0787);
  \draw[->][ForestGreen] (B4) -- (-0.08, 0.063);
  \draw[->][ForestGreen] (B5) -- (-0.38, 0.0547);
  
  \draw (-2, 1/4) -- (-1, 0);
  \draw (1, 0) -- (0, 1);
\draw[crimson] (-1, 0) -- (1, 0);
  
  \draw[->][blue] (B0) -- (B1);
  \draw[->][blue] (B1) -- (B2);
  \draw[->][blue] (B2) -- (B3);
  \draw[->][blue] (B3) -- (B4);
  \draw[->][blue] (B4) -- (B5);
  \draw[->][blue] (B5) -- (B6);
  \draw[->][blue] (B6) -- (B7);
  \draw[->][blue] (B7) -- (B8);
  \draw[->][blue] (B8) -- (B9);
  
\node[fill=black,circle,label=below:
  $A_0$,inner sep=0.9pt] at (A0) {};
\node[fill=black,circle,label=below:
  $A_1\;\;$,inner sep=0.9pt] at (A1) {};
\node[fill=black,circle,label=below:
  $\;\;A_2$,inner sep=0.9pt] at (A2) {};

\node[fill=black,circle,label=below:
  $C_0$,inner sep=0.9pt] at (C0) {};
\node[fill=black,circle,label=above:
  $\;\;C_1$,inner sep=0.9pt] at (C1) {};
\node[fill=black,circle,label=above:
  $C_2\;\;$,inner sep=0.9pt] at (C2) {};
  
  \node[fill=black,circle,label=
  $B_0$,inner sep=0.9pt] at (B0) {};
  \node[fill=black,circle,label=left:{$B_1=V_0$},inner sep=0.9pt] at (B1) {};
  \node[fill=black,circle,label=above:$\quad\;B_2$,inner sep=0.9pt] at (B2) {};
  \node[fill=black,circle,label=right:$B_3$,inner sep=0.9pt] at (B3) {};
  \node[fill=black,circle,label=above:$B_4\;$,inner sep=0.9pt] at (B4) {};
  \node[fill=black,circle,label=left:$B_5$,inner sep=0.9pt] at (B5) {};
  \node[fill=black,circle,inner sep=0.9pt] at (B6) {};
  \node[fill=black,circle,inner sep=0.9pt] at (B7) {};
  \node[fill=black,circle,inner sep=0.9pt] at (B8) {};
  \node[fill=black,circle,inner sep=0.9pt] at (B9) {};
  
  \node[fill=black,circle,
  inner sep=0.9pt] at (1,0) {};
  \node[below] at (1,0) {\begin{tabular}{c} $V_1=V_2$\\$=V_5$\end{tabular}$\quad$};
  \node[fill=black,circle,label=below:{$V_3=V_4$},inner sep=0.9pt] at (-1,0) {};
  
 \end{tikzpicture}
 \caption{\emph{Stretched for visualization purposes.} The constraint set (in black), polygonal sketches of the objective function (in orange) built using the open-loop strategy~\ref{step:open2}, the solution set (in red), gradient directions (in green), 
 and the trajectory (in blue) of 
 $(B_k)_{k\in\mathbb{N}}$. 
 For all $k\geq1$, $B_{k+1}\in\left[B_k,V_k\right]$ where $V_k=V_{k-1}$ or $V_k=-V_{k-1}$, depending on whether $(B_k)_{k\in\mathbb{N}}$ has just crossed the vertical band delimited by the dashed lines or not.}
 \label{fig:open}
\end{figure}
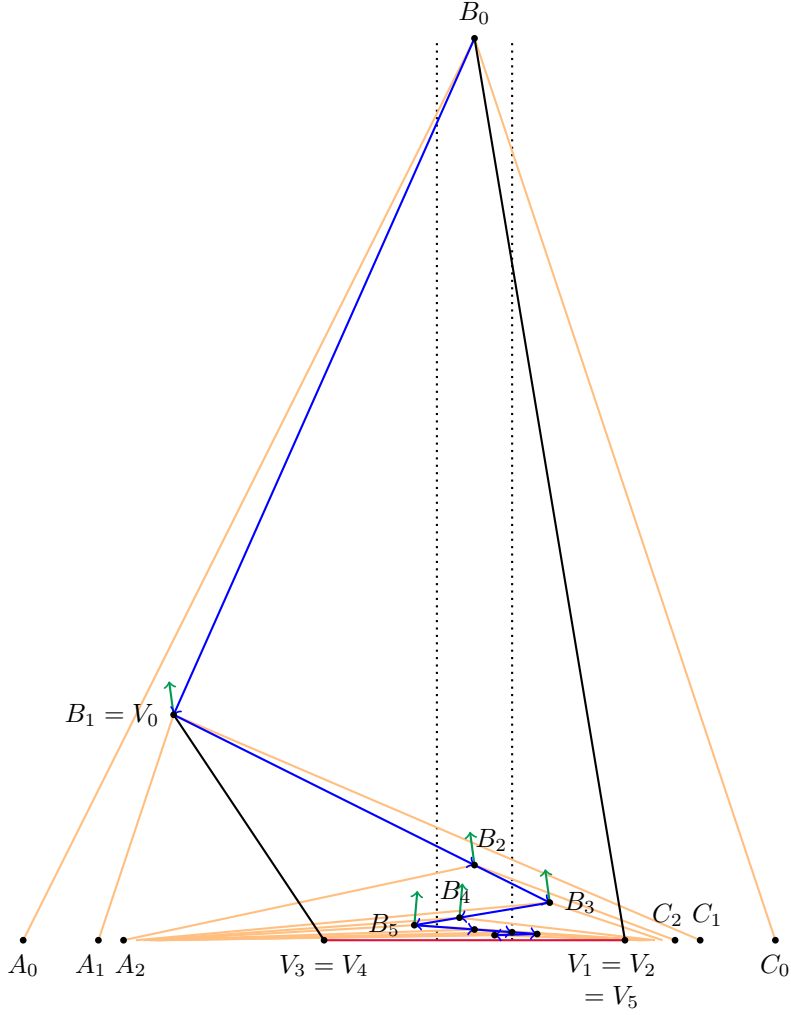

By construction, for all $k\in\mathbb{N}$, 
\begin{align*}
 \mathcal{K}_{\mathcal{P}_k}B_k\cap(\mathbb{R}_-^*\times\mathbb{R}_+^*)\neq\varnothing
 \quad\text{and}\quad
 \mathcal{K}_{\mathcal{P}_k}B_k\cap(\mathbb{R}_+^*)^2\neq\varnothing.
\end{align*}
Thus, we can choose $(u_k)_{k\in\mathbb{N}}$ such that for all $k\in\mathbb{N}$, 
\begin{align*}
 u_k\in\mathcal{K}_{\mathcal{P}_k}B_k
 \quad\text{and}\quad
 V_k\in\argmin_{v\in\mathcal{C}}\,\langle v,u_k\rangle.
\end{align*}
Let $x_0=B_0$. By induction, $x_t=B_t$ for all $t\in\mathbb{N}$. Therefore, $(x_t)_{t\in\mathbb{N}}$ crosses the vertical band $\{x\in\mathcal{C}\mid|\langle x,e_1\rangle|\leq1/4\}$ indefinitely, i.e., $\operatorname{card}\{t\in\mathbb{N}\mid\langle x_t,e_1\rangle\leq-1/4\}=+\infty$ and $\operatorname{card}\{t\in\mathbb{N}\mid\langle x_t,e_1\rangle\geq1/4\}=+\infty$. Since $\gamma_t<1$ for all $t\geq1$, $\langle x_t,e_2\rangle>0$ for all $t\in\mathbb{N}$, so $(x_t)_{t\in\mathbb{N}}$ does not converge.

\subsection*{Acknowledgment}

The authors acknowledge the support of the AI Interdisciplinary Institute ANITI funding, through the French ``Investments for the Future -- PIA3'' program under the grant agreement ANR-19-PI3A0004, Air Force Office of Scientific Research, Air Force Material Command, USAF, under grant numbers FA9550-19-1-7026, and ANR MaSDOL 19-CE23-0017-0. J\'er\^ome Bolte also acknowledges the support of ANR Chess, grant ANR-17-EURE-0010, and TSE-P.

\bibliographystyle{abbrv}
{\small\bibliography{biblio}}

\end{document}